\documentclass[12pt]{amsart}
\usepackage{amssymb, latexsym}
\usepackage{amsthm}
\usepackage{amsmath}
\usepackage{amsfonts}
\usepackage[margin=1.25in]{geometry}
\usepackage[all,cmtip]{xy}

\theoremstyle{plain}
\newtheorem{theorem}{Theorem}[section]
\newtheorem{corollary}[theorem]{Corollary}
\newtheorem{proposition}[theorem]{Proposition}
\newtheorem{lemma}[theorem]{Lemma}

\theoremstyle{definition}
\newtheorem{definition}[theorem]{Definition}
\newtheorem{remark}[theorem]{Remark}
\newtheorem{example}[theorem]{Example}

\numberwithin{equation}{section}

\usepackage{hyperref}
\usepackage[all]{xy}

\newcommand{\CMLt}{\text{CML}_3}

\newcommand{\CMLtThm}{\emph{CML}_3}

\newcommand{\ld}{\backslash_{\circ}}
\newcommand{\rd}{/^{*}}

\newcommand{\rack}{\triangleright^r}

\usepackage[dvipsnames]{color}

\begin{document}

\title[Dihedral solutions of the SYBE]{Dihedral solutions of the set theoretical Yang-Baxter equation}

\author[A. W. Nowak]{Alex W. Nowak }
\address{Department of Mathematics\\
Howard University\\
Washington, D.C. 20059, U.S.A.}

\email{alex.nowak@howard.edu}

\author[A. Zamojska-Dzienio]{Anna Zamojska-Dzienio}
\address{Faculty of Mathematics and Information Science\\ Warsaw University of Technology\\ 00-661 Warsaw, Poland}
\email{anna.zamojska@pw.edu.pl}

\keywords{Quantum Yang-Baxter equation, set theoretical Yang-Baxter equation, nondegenerate solution, latin solution, Bruck loop, commutative Moufang loop, symmetric space}
\subjclass[2010]{16T25, 20N05}

\begin{abstract}
We introduce the notion of a \emph{braided dihedral set} (BDS) to describe set-theoretical solutions of the Yang-Baxter equation (YBE) that furnish representations of the infinite dihedral group on the Cartesian square of the underlying set. BDS which lead to representations of the symmetric group on three objects are called \emph{braided triality sets} (BTS). Basic examples of BDS come from symmetric spaces.
We show that Latin BDS (LBDS) can be described entirely in terms of involutions of uniquely 2-divisible Bruck loops. We show that isomorphism classes of LBDS are in one-to-one correspondence with conjugacy classes of involutions of uniquely 2-divisible Bruck loops. We describe all LBDS of prime, prime-square and 3 times prime-order, up to isomorphism. Using \texttt{GAP}, we enumerate isomorphism classes of LBDS of orders 27 and 81.
Latin BTS, or LBTS, are shown to be in one-to-one correspondence with involutions of commutative Moufang loops of exponent 3 (CML3), and, as with LBDS, isomorphisms classes of LBTS coincide with conjugacy classes of CML3-involutions. We classify all LBTS of order at most 81.
\end{abstract}

\maketitle

\section{Introduction}\label{Sec:Intro}
\subsection{Background and motivation}\label{SubSec:Background}

The quantum Yang-Baxter equation, which first appeared in the field of statistical mechanics, has proven to be an object of interest in many areas of mathematics, including the representation theory of quantum groups \cite{Drinfeld88} and knot theory \cite{Jones, Turaev}. We study its set-theoretical articulation, which was first proposed by Drinfeld \cite{Drinfeld}. Namely, a function of sets $r:Q^2\to Q^2$ is a \emph{set-theoretical solution to the quantum Yang-Baxter equation} if

\begin{equation}\label{Eq:SetBraidEqn}
(r\times 1_Q)(1_Q\times r)(r\times 1_Q)=(1_Q\times r)(r\times 1_Q)(1_Q\times r).
    \end{equation}
The term \emph{solution} refers to a set-map pair $(Q, r)$ such that $r$ adheres to \eqref{Eq:SetBraidEqn}. We also use the term \emph{braided set} to refer to such a pair.  We call a solution $(Q, r)$ \emph{bijective} if $r$ is bijective. A solution is \emph{involutive} if $r^2=1_{Q^2}$. It is \emph{idempotent} if $r^2=r$. More generally, $(Q, r)$ has \emph{finite order} if there is a positive integer $n$ and nonnegative integer $i$ for which $r^{n+i}=r^i$.

With respect to a binary operation $*$ on $Q$ and each $q\in Q,$ define \emph{left and right multiplication maps}:
\begin{align*}
    L^*_q&:x\mapsto q*x,\\
    R^*_q&:x\mapsto x*q.
\end{align*}
Now, given a solution $r:Q^2\to Q^2$, define binary operations $\circ$ and $\bullet$ on $Q$ such that $r(x, y)=(x\circ y, x\bullet y)$. If for each $q\in Q$, $L^\circ_q$ is bijective, then we say $(Q, r)$ is \emph{left nondegenerate}, while if each $R^\bullet_q$ is bijective, then $(Q, r)$ is \emph{right nondegenerate}. A solution is \emph{nondegenerate} if it is left and right nondegenerate. If $(Q, r)$ is left nondegenerate and each $R^\circ_q$ is also bijective, we say $(Q, r)$ is a \emph{Latin solution}.

While nondegenerate, involutive solutions have garnered the most interest (\cite{Etingov99, Rump} are among the most prominent works in this area), nondegenerate, bijective solutions are well-studied too \cite{Lu, Soloviev}. Now, with the combinatorial and semigroup-theoretic utility of idempotent solutions well-established \cite{Colazzo24, LebedPl}, there is a burgeoning theory of left nondegenerate, not necessarily bijective solutions \cite{Catino, Colazzo}. Latin solutions are considered in \cite{BKSV, StanVoj}

The solutions that initially motivated this investigation are nondegenerate, Latin. Moreover, they are bijective and have finite order, but they are not involutive, and do not necessarily need to be defined on a finite set. They were introduced by Smith \cite[Thm.~5.7]{Sm16}. They are defined as
\begin{equation}\label{Eq:SMith'sSol}
    r:(x, y)\mapsto (-x+y, -x),
\end{equation}
where $(Q, +)$ is a (not necessarily finite) commutative Moufang loop of exponent $3$, or $\CMLt$ ---cf. Section \ref{SubSec:Loops} below for the definition of $\CMLt$, and if unfamiliar with loop theory, for now, you can take $(Q, +)$ to be an elementary abelian $3$-group.
 If we let $\tau:(x, y)\mapsto(y, x)$ denote the trivial solution, then it so happens that
 \begin{equation}\label{Eq:DihedRel}
     \tau^2=(\tau r)^2=1_{Q^2},
 \end{equation}
 so that $Q^2$ is a $D_\infty$-set under the action of $\langle r, \tau\rangle\leq S_{Q^2}$, where $D_\infty$ denotes the infinite dihedral group.  We also have
 \begin{equation}\label{Eq:TrialityRelation}
     r^3=1_{Q^2},
 \end{equation}
making $Q^2$ an $S_3$-set too.

\begin{remark}\label{Rmk:DuallyInvolutive}
   For a function $\rho:Q^2\to Q^2$, let $\rho^{ij}:Q^3\to Q^3$ act as $\rho$ in the $i$th and $j$th components of $Q^3$ while fixing the remaining component. We note here that $r$ is a solution of \eqref{Eq:SetBraidEqn} if and only if the map $\rho:=\tau r$  is a solution to the quantum Yang-Baxter equation:
    \begin{equation}\label{Eq:QYBE}
        \rho^{12}\rho^{13}\rho^{23}=\rho^{23}\rho^{13}\rho^{12}.
    \end{equation}
   It follows that $r$ is a dihedral solution if and only if $\rho=\tau r$ is an involutive map satisfying \eqref{Eq:QYBE}. In this sense, the solutions we study in this paper are dual to the famous involutive solutions of \eqref{Eq:SetBraidEqn}.
    \end{remark}
For the sake of our discussion in the next section, we mention here that another class of solutions exhibiting $S_3$ symmetry on $Q^2$ includes those of the form
\begin{equation}\label{Eq:CMLDerSol}
    r:(x, y)\mapsto (-x-y, x).
\end{equation}
The operation $-x-y$ over a $\CMLt$ is extensively studied in design theory, for it gives rise to a famous class of Steiner triple systems, the so-called Hall triple systems \cite{Hall}. In the study of the SYBE, solutions, which, like \eqref{Eq:CMLDerSol}, are of the form $(x, y)\mapsto(x\circ y, x)$ are called \emph{derived solutions} and are of particular interest. See \cite{Soloviev} for the first in-depth treatment of derived solutions and the introduction of \cite{Lebed} for an excellent account of their significance in the contemporary theory.

As it turns out, there are generalizations of \eqref{Eq:SMith'sSol} and \eqref{Eq:CMLDerSol} from $\CMLt$'s to Bruck loops of odd exponent ---for those uninitiated in loop theory, think abelian groups of odd exponent for now. Namely, whenever $(Q, +)$ is a Bruck loop of odd exponent,
\begin{equation}\label{Eq:GenofSmSoln}
     (x, y)\mapsto (2x+y, -x)
\end{equation}
and
\begin{equation}\label{Eq:GenofDerSoln}
     (x, y)\mapsto (2x-y, x)
\end{equation}
are solutions to \eqref{Eq:SetBraidEqn} (cf. Theorem \ref{Thm:LBDSCharacterization} below), and while they do not necessarily induce $S_3$-symmetry on $Q^2$, they do still adhere to \eqref{Eq:DihedRel}. In fact, if the operation $+$ in \eqref{Eq:GenofSmSoln} or \eqref{Eq:GenofDerSoln} is an abelian group operation of exponent $n$, it is easy to show that $r^n=1$ as well, and so $Q^2$ becomes a $D_n$-set under the action of $\langle r, \tau\rangle.$

\subsection{Plan of the paper}\label{SubSec:Plan}
This paper explores the restrictions that the dihedral relations \eqref{Eq:DihedRel} impose on a solution $(Q, r)$, and in the process, we introduce a new class of SYBE solutions that we call \emph{braided dihedral sets} (BDS). We also consider BDS adhering to \eqref{Eq:TrialityRelation}, which we call \emph{braided triality sets} (BTS). The majority of the paper is devoted to Latin BDS and BTS, or LBDS and LBTS. We characterize them in terms of Bruck loops and $\CMLt$'s, and, thus, we spend Section \ref{Sec:QugpsandLoops} establishing the necessary background on quasigroups and loops. Section \ref{Sec:BDS} provides some initial examples and connects BDS to previously-studied classes of solutions.

In Section \ref{Sec:LBDS}, we tie LBDS $(Q, r)$ to a well-studied class of loops called uniquely $2$-divisible Bruck loops. In fact, Theorem \ref{Thm:LBDSCharacterization} says that $(Q, r)$ is an LBDS if and only if there is a uniquely $2$-divisible Bruck loop $(Q, +)$ with involutory automorphism $S$ such that $r(x, y)=(2x-y^S, x^S)$. Theorem \ref{Thm:LBDSIsom} then proves that the isomorphism problem for LBDS is equivalent to determining conjugacy classes of involutions in automorphism groups of uniquely $2$-divisible Bruck loops. We also consider the effect of the so-called \emph{right-divisional squaring map} $\text{Sq}_{/^\circ}:x\mapsto (R^{\circ}_{x})^{-1}(x)$ on the structure of $(Q, r)$. Theorem \ref{Thm:LBDSSplitExt} states that if $\text{Sq}_{/^\circ}$ is endomorphic with respect to $\circ$, then $(Q, r)$ is a (quasigroup and loop) split extension of a type-\eqref{Eq:GenofDerSoln} solution by a type-\eqref{Eq:GenofSmSoln} solution. In Section \ref{SubSec:SmallOrderLBDS} we explicitly describe all LBDS of orders $p, p^2,$ and $3p$ up to isomorphism, where $p$ stands for an arbitrary odd prime. We also enumerate the isomorphism classes of LBDS of orders $27$ and $81$.

Section \ref{Sec:LBTS} specializes the results of Section \ref{Sec:LBDS} to LBTS. We show that $(Q, r)$ is an LBTS if and only if $r(x, y)=(-x-y^S, x^S)$, where $(Q, +)$ is a $\CMLt$ with involutive automorphism $S$ (Corollary \ref{Cor:LBTSCML}). With LBTS, the map $\text{Sq}_{/^\circ}$ imposes even more structure on $(Q, r)$. Without any assumptions on whether or not it is endomorphic, we leverage $\text{Sq}_{/^\circ}$ in proving Theorem \ref{Thm:LBTSFactorization}, which shows that any LBTS $(Q, r)$ has a $\CMLt$-factorization $Q=P+R$, where $P$ is a solution of the form \eqref{Eq:CMLDerSol} and $R$ is a \eqref{Eq:SMith'sSol}-type solution. Theorem \ref{Thm:LBTS81} concludes the paper by explicitly describing all LBTS of order $n\leq 81$.

\section{Quasigroups, loops, and an alternative formulation of the SYBE}\label{Sec:QugpsandLoops}
\subsection{Quasigroups}\label{SubSec:QugpsLSquares}

\begin{definition}\label{Def:Qugp}
Let $Q$ be a set, and consider three binary operations on $Q$: $*$ denoting \textit{multiplication}, $/^*$ \textit{right division}, and $\backslash_*$ \textit{left division}.  Consider also the following identities in these operations:
\begin{align*}
\text{(IL)}\phantom{A} y\backslash_*(y* x)=x, && \text{(IR)}\phantom{A}x= (x* y)/^* y,\\
\text{(SL)}\phantom{A} y*(y\backslash_* x)=x, && \text{(SR)}\phantom{A}x= (x/^* y)* y.
\end{align*}

\noindent If (IL) and (SL) hold, then $(Q, *, \backslash_*)$ is a \emph{left quasigroup}.  If (IR) and (SR) hold, then $(Q, *, /^*)$ is a \emph{right quasigroup}.  If all four identities hold, then $(Q, *, \backslash_*, /^*)$ is a \emph{quasigroup}.
\end{definition}

\noindent Note that if we divide (SL) on the right by $y\backslash_*x$ and (SR) on the left by $x/^*y$ we get two identities:

\begin{equation}\label{Eq:DivTents}
   x/^*(y\backslash_*x)=(x/^*y)\backslash_*x=y.
\end{equation}

\begin{remark}
\begin{itemize}
    \item[(a)] Given a magma (set under binary operation) $(Q, *)$, $*$ is a left (resp. right) quasigroup multiplication if and only if, for all $q\in Q$, the left (right) translation map $L^{*}_q:Q\to Q; x\mapsto q* x$ ($R^*_q:Q\to Q; x\mapsto x* q$) is invertible.  Axiom (IL) (resp. (IR)) says that each left (right) translation is injective, while (SL) ((SR)) ensures surjectivity of each left (right) translation.
    \item[(b)] Going forward, it will be convenient to use concatenation to denote certain instances of the multiplication, $*$, in a quasigroup $(Q, *, \backslash_*, /^*)$. We will abbreviate $x* y$ to $xy$, and concatenation will bind more tightly than $*$. For example, $(x* y)* z$ will abbreviate as $xy* z.$
    \item[(c)] A mapping $h\colon (Q, \circ, \ld, /^\circ)\to (P, \bullet, \backslash_\bullet, /^\bullet)$ between two quasigroups is a \emph{quasigroup homomorphism} if it preserves the three quasigroup operations. However, preserving one operation implies that the property holds for the other two.
    \item[(d)] A triple of maps between quasigroups $(f, g, h):(Q, \circ, \ld, /^\circ)\to (P, \bullet, \backslash_\bullet, /^\bullet)$ is called a \emph{homotopy} if $f(x)\bullet g(y)=h(x\circ y)$ for all $x, y\in Q$. If each of the three maps is a bijection, $(f, g, h)$ is called an \emph{isotopy}.
\end{itemize}
\end{remark}

\begin{definition}\label{Def:InfixSoltns}
Let $Q$ be a nonempty set, and $$r:Q^2\to Q^2; (x, y)\mapsto (x\circ y, x\bullet y)$$ an invertible map. Let $\tau(x, y)=(y, x)$.  The pair $(Q, r)$ is said to be
\begin{itemize}
    \item[(a)] a \emph{braided set}, or a \emph{solution} if $r$ adheres to \eqref{Eq:SetBraidEqn}; that is, if
    \begin{equation}\label{Eq:YB1}
    x\circ (y\circ z)=(x\circ y)\circ ((x\bullet y)\circ z),
\end{equation}

\begin{equation}\label{Eq:YB2}
    (x\circ y)\bullet ((x\bullet y)\circ z)=(x\bullet (y\circ z))\circ (y\bullet z),
\end{equation}

\begin{equation}\label{Eq:YB3}
    (x\bullet y) \bullet z=(x\bullet(y\circ z))\bullet(y\bullet z),
\end{equation}
hold in Q;
\item[(b)] \emph{left nondegenerate} if $\circ$ gives rise to a left quasigroup $(Q, \circ, \ld)$, \emph{right nondegenerate} if $\bullet$ furnishes a right quasigroup $(Q, \bullet, /^\bullet)$, and \emph{nondegenerate} if it is both left and right nondegenerate;
\item[(c)] \textit{Latin} if $\circ$ gives rise to a quasigroup $(Q, \circ, \ld, /^\circ)$;
\item[(d)] \textit{dihedral} if $(\tau r)^2=1_{Q^2}$; that is, if
\begin{equation}\label{Eq:Di1}
    (x\bullet y)\bullet (x\circ y)=x,
\end{equation}

\begin{equation}\label{Eq:Di2}
    (x\bullet y)\circ (x\circ y)=y,
\end{equation}
each hold in $Q$;
\item[(e)] a \emph{triality set} if $(\tau r)^2=r^3=1_{Q^{2}}$; that is, if \eqref{Eq:Di1}-\eqref{Eq:Di2} hold, and
\begin{equation}\label{Eq:Tri1}
    (x\circ y)\circ (x\bullet y)=y\bullet x,
\end{equation}

\begin{equation}\label{Eq:Tri2}
    (x\circ y)\bullet (x\bullet y)=y\circ x,
\end{equation}
each hold in $Q$.
\end{itemize}
\end{definition}

\begin{remark}\label{Rmk:Birack}
Nondegenerate braided sets are known as \emph{biracks} and they were introduced in \cite{Fenn} in the context of the virtual knot theory \cite{Kauffman}. It was proved in \cite{biquandles} that the following identities:
\begin{equation}\label{E:Sta_6}
 (x\setminus_\circ x)/^\bullet (x\setminus_\circ x)=x .
\end{equation}
\begin{equation}\label{E:Sta_7}
 (x/^\bullet x)\setminus_\circ (x/^\bullet x)=x.
\end{equation}
are equivalent in any birack. Then, a birack which satisfies \eqref{E:Sta_6} or \eqref{E:Sta_7} is called a \emph{biquandle}.
 If one considers the squaring mappings with respect to operations $\setminus_\circ$ and $/^\bullet$ (also known as \emph{diagonal mappings}), then identities \eqref{E:Sta_6} and \eqref{E:Sta_7} say that in any biquandle mappings $\text{Sq}_{\setminus_\circ}$ and $\text{Sq}_{/^\bullet}$ are mutually inverse. In \cite{JedPil} it was shown that in any nondegenerate solutions of \eqref{Eq:SetBraidEqn} diagonal mappings are commuting bijections, but not necessarily mutually inverse.  However, Latin braided dihedral sets are biquandles satisfying, in addition, $\text{Sq}_{\setminus_\circ}=\text{Sq}_{/^\bullet}$ (see Lemmas \ref{Lmm:STmutinv} and \ref{Lmm:DihedSinv} in the Appendix). 
\end{remark}
There are some so-called \emph{symmetry classes} of left and right quasigroups that will come up repeatedly that we shall now review. A magma satisfying
\begin{equation}\label{Eq:LS}
    x* xy=y
\end{equation} is said to be \emph{left symmetric}. A left symmetric magma is necessarily a left quasigroup, for $(L^*_x)^{-1}=L^*_x$. Dually, magmas in which the right translations are involutive, so that
\begin{equation}\label{Eq:RS}
    y x* x=y
\end{equation}
holds, are \emph{right symmetric} right quasigroups. Any magma in which
\begin{equation}\label{Eq:SS1}
    x* yx=y
\end{equation}
holds is a quasigroup, for \eqref{Eq:SS1} implies $L^*_x=(R^*_x)^{-1}$. Equivalently, $R^*_x=(L^*_x)^{-1}$ translates to
\begin{equation}\label{Eq:SS2}
    xy * x=y.
\end{equation}
Quasigroups satisfying the equivalent identities \eqref{Eq:SS1}-\eqref{Eq:SS2} are \emph{semisymmetric}.
Note that Lemma \ref{Lmm:CommiffLCisRS} below may be used to show that any two of \eqref{Eq:LS}-\eqref{Eq:SS1} imply the third. Quasigroups in which these three identities hold are called \emph{totally symmetric} quasigroups. Totally symmetric quasigroups are also commutative.

We conclude this subsection by collecting a couple of elementary quasigroup-theoretic results that will serve us later. The first relates the potential commutativity of a quasigroup to the potential right symmetry of the left division operation.

\begin{lemma}\label{Lmm:CommiffLCisRS}
With respect to a quasigroup $(Q, *, \backslash_*, /^*)$, all of the following identities are equivalent:
\begin{itemize}
    \item[(a)] $xy=yx$;
    \item[(b)] $x/^*y=y\backslash_*x$;
    \item[(c)] $(y\backslash_*x)\backslash_*x=y$.
\end{itemize}
\end{lemma}

\begin{proof}
    $\text{(a)}\implies \text{(b)}$ By (IL) and (SR), $x/^*y=y\backslash_*(y *(x/^*y))=y\backslash_*((x/^*y)* y)=y\backslash_*x.$\\
    $\text{(b)}\implies \text{(a)}$ By (SL) and (IR), $xy=y*(y\backslash_*(xy))=y* (xy/^*y)=yx$.\\
    $\text{(b)}\implies \text{(c)}$ By \eqref{Eq:DivTents}, $(y\backslash_*x)\backslash_*x=x/^*(y\backslash_*x)=y.$ \\
    $\text{(c)}\implies \text{(b)}$ By \eqref{Eq:DivTents}, $x/^*y=x/^*((y\backslash_*x)\backslash_*x)=y\backslash_*x.$
\end{proof}

\begin{lemma}\label{Lmm:IPMapInvo}
    Let $(Q, *, \backslash)$ be a left quasigroup, and define $x^S:=x\backslash_*x$. If $x^S* xy=y$, then the following identities hold in $(Q, *, \backslash_*)$:
    \begin{itemize}
        \item[(a)] $S^2=1_Q$
        \item[(b)] $x* x^S y=y.$
    \end{itemize}
\end{lemma}

\begin{proof}
    Note $x^S\backslash_*y=x^S\backslash_*(x^S* xy)=xy$ for any $x, y\in Q$. In particular $x^{S^{2}}=x^S\backslash_*x^S=x* x^S=x(x\backslash_*x)=x$, proving (a). To see that (b) holds, note $x^Sy=x^S(x(x\backslash_*y))=x\backslash_*y$, so we conclude $x* x^Sy=x(x\backslash_*y)=y.$
\end{proof}

\subsection{Idempotence and distributivity of operations}
We shall denote the squaring map with respect to an operation $*$ by $\text{Sq}_*(x)=x* x$. When context permits, we may also use exponent notation $x^2=x*x$. When exponentiation appears in an expression involving multiple quasigroup operations, it is understood that these powers are with respect to the multiplication. For example,  $x^2/y=(x* x)/y$ in $(Q, *, \backslash_*, /^*)$.
An element $x$ of a magma $(Q, *)$ is called an \emph{idempotent} of the magma if
\begin{equation}
    x^2=x.
\end{equation}
If each element of $Q$ is an idempotent, then $(Q, *)$ is said to be \emph{idempotent}.

If one expresses the quasigroup axioms (IL)-(SR) in terms of a single variable $x$, the next lemma follows.

\begin{lemma}\label{Lmm:IdempQugp}
If an element $x$ in a quasigroup $(Q, *, \backslash_*, /^*)$ is idempotent with respect to one of the operations, then it is idempotent with respect to all of them.
\end{lemma}

A multiplication $*$ is a \emph{left distributive} operation if for all $x, y, z\in Q$,

\begin{equation}\label{Eq:LD}
    x*(y*z)=(x*y)*(x* z).
\end{equation}

A left distributive, left quasigroup is called a \emph{rack}. Note that a map $r(x, y)=(x\circ y, x)$ is a left nondegenerate, derived solution (cf. Section \ref{SubSec:Background}) if and only if $(Q, \circ)$ is a rack. Moreover, to any left nondegenerate solution $r:(x, y)\mapsto(x\circ y, x\bullet y)$, one may associate the \emph{derived solution} $(x, y)\mapsto(x\rack y, x)$, where
\begin{equation}\label{Eq:StrRack}
    x\rack y=x\circ(y\bullet(y\ld x)),
\end{equation}
and $(Q, \rack)$ is called the \emph{derived rack} of the solution \cite[Thm.~1.1]{Soloviev}.

An idempotent rack is called a \emph{quandle}. Quandles and racks which are two-sided quasigroups are referred to as \emph{Latin}. A left symmetric quandle is often referred to as an \emph{involutive quandle} or a \emph{kei}. Whereas quandles provide algebraic invariants of oriented knots, the left-symmetric  identity makes keis suitable for producing algebraic invariants of unoriented knots \cite{Joyce}.

\subsection{SYBE solutions coming from LF-left quasigroups} \label{SubSec:LFQgps}
Until this point, we have been reviewing some standard terms and proving folkloric results from quasigroup theory. In this section, however, we will provide some original observations and new results that clarify the role that LF-left quasigroups play in generating braided sets.

For a left nondegenerate map $r:(x, y)\mapsto (x\circ y, x\bullet y)$, it will be convenient for us to abbreviate $\text{Sq}_{\ld}(x)=x\ld x$ to simply $S$. Moreover, we employ the exponential notation $x^S:=x\ld x.$ The variety of \emph{LF-left quasigroups} is generated by the following variation of \eqref{Eq:LD}:

\begin{equation}\label{Eq:LFNew}
x\circ(y\circ z)=(x\circ y)\circ (x^S\circ z).
\end{equation}

\begin{remark}\label{Rmk:LF}
   A (two-sided) quasigroup satisfying \eqref{Eq:LFNew} is called a left F-quasigroup (or an LF-quasigroup). Dually, one defines right F-quasigroups via
   \begin{equation}\label{Eq:RF}
    (z* y)* x=(z* (x/^*x))* (y* x).
    \end{equation}
   If a quasigroup satisfies both \eqref{Eq:LFNew} and \eqref{Eq:RF}, it belongs to the variety of F-quasigroups, which is one of the oldest studied varieties of quasigroups, introduced already by Murdoch in 1939, although he did not use this name. For details and comprehensive references, see \cite{Shcherb} or \cite[Chapter 6]{Shcherb_book}.
  \end{remark}

The following lemma generalizes a known property of LF-quasigroups to LF-left quasigroups. In fact, the proof given in \cite[Lmm.~2.1]{Shcherb} does not rely on a right division, so we will simply reproduce it below.

\begin{lemma}\label{Lmm:SEndo}
    In an LF-left quasigroup, $(Q, \circ, \ld)$, the map $x^S=x\ld x$ is an endomorphism.
\end{lemma}

\begin{proof}
    Since $y\circ y^S=y,$ \eqref{Eq:LFNew} implies $x\circ y=x\circ(y\circ y^S)=(x\circ y)\circ(x^S\circ y^S)$, and, thus, $(x\circ y)^S=(x\circ y)\ld(x\circ y)=(x\circ y)\ld((x\circ y)\circ (x^S\circ y^S))=x^S\circ y^S$.
\end{proof}

The following proposition illustrates that LF-left quasigroups are to left nondegenerate solutions of the form $(x, y)\mapsto (x\circ y, x^S)$ as racks are to left nondegenerate derived solutions.

\begin{proposition}\label{Prop:F-qgpiffBraid}
A left nondegenerate map of the form $(x, y)\mapsto (x\circ y, x^S)$ is a braiding if and only if $(Q, \circ, \ld)$ is an LF-left quasigroup.
\end{proposition}

\begin{proof}
If $(x, y)\mapsto (x\circ y, x^S)$ is a left nondegenerate braiding, then \eqref{Eq:LFNew} follows from \eqref{Eq:YB1}.

Conversely, assume $(Q, \circ, \ld)$ is an LF-left quasigroup and define $r: (x, y)\mapsto (x\circ y, x^S)$. Then \eqref{Eq:YB1} follows by definition. Moreover, $\eqref{Eq:YB3}$ holds trivially. Finally, \eqref{Eq:YB2} translates to $(x\circ y)^S=x^S\circ y^S$, which is established in Lemma \ref{Lmm:SEndo} Hence, \eqref{Eq:YB1}-\eqref{Eq:YB3} are satisfied and $(Q, r)$ is braided.
\end{proof}

\begin{proposition}
\label{Prop:F-qugpsiffSqmap}
Assume that $(x, y)\mapsto(x\circ y, x\bullet y)$ is a Latin braiding. Then $(Q, \circ, \ld, /^\circ)$ is an LF-quasigroup if and only if $x\bullet y=x^S$.
\end{proposition}

\begin{proof}
Suppose $(x, y)\mapsto (x\circ y, x\bullet y)$ is a Latin braiding and that $(Q, \circ, \ld, /^\circ)$ is an LF-quasigroup.
Upon dividing \eqref{Eq:LFNew} on the left by $x\circ y$, we have
    \begin{align*}
        x^S\circ z&=(x\circ y)\ld(x\circ(y\circ z))\\
        &=(x\circ y)\ld((x\circ y)\circ((x\bullet y)\circ z))\\
        &=(x\bullet y)\circ z,
    \end{align*}
where the second equality comes from \eqref{Eq:YB1}. Since $\circ$ accompanies a right division, we can cancel the $z$ in the first and last equalities above, yielding $x\bullet y=x^S$.

The converse follows directly from Proposition \ref{Prop:F-qgpiffBraid}.
\end{proof}

\begin{remark}\label{Rmk:LFNOTRIGHT}
We remark on the necessity of the Latin property in Proposition \ref{Prop:F-qugpsiffSqmap}. Now, if we assume $r$ is merely a left nondegenerate braiding and not Latin, then $x\bullet y=x^S$ does in fact imply $(Q, \circ, \ld)$ is an LF-left quasigroup. This is a consequence of Proposition \ref{Prop:F-qgpiffBraid}. However, there are nondegenerate braidings coming from LF-left quasigroups such that $x\bullet y\neq x^S.$ In fact, any right derived braiding $(x, y)\mapsto (y, x\bullet y)$ furnishes such an example.
\end{remark}

\subsection{Loops}\label{SubSec:Loops}
A quasigroup with a multiplicative identity is called a \emph{loop}. We write $(Q, *, e)$ to describe the structure of a loop. Note that the loop axioms are not enough to ensure the existence two-sided inverses. That is, for some $x\in Q$ there may exist $y\neq z$ such that $xy=zx=e$. When a loop does have two-sided inverses, we write $x^{-1}$ to denote that unique element for which $x^{-1}x=xx^{-1}=e$. A loop with two-sided inverses in which
\begin{equation}\label{Eq:AIP}
    (xy)^{-1}=x^{-1}y^{-1},
\end{equation}
holds for all $x$ and $y$ is called an \emph{automorphic inverse property loop} (AIP).
A loop $Q$ is said to have the \emph{left inverse property} (we call $Q$ an LIP loop) if for every $x\in Q$, there is a two-sided inverse, $x^{-1}$ such that for all $y\in Q$,
 \begin{align}
    x^{-1}(xy)=x(x^{-1}y)=y. \label{Eq:LIP}
\end{align}
The two-sided inverse of a \emph{right inverse property loop} (RIP loop) adheres to
 \begin{align}
y=(yx^{-1})x=(yx)x^{-1}. \label{Eq:RIP}
\end{align}
\emph{Inverse property loops} (IP loops) possess the left and right inverse properties. A loop in which
\begin{equation}\label{Eq:LBolLaw}
    x(y(xz))=(x(yx))z
\end{equation}
for all $x, y, z\in Q$ is called a \emph{(left) Bol loop}. Bol loops are LIP loops \cite[Thm.~2.1]{ROB}. Moreover, Bol loops are power-associative, i.e., every element generates a group. In fact, by \cite[Lmm.~2.1]{ROB} for all $x, y\in Q$ and all integers $n$,
\begin{equation}\label{Eq:PowerAssoc}
    y^nx=y^{n-1}y* x=y^{n-1}* yx.
\end{equation}

If a Bol loop satisfies the AIP, it is called a \emph{Bruck loop}. We will use an additive notation for Bruck loops: $(Q, +, 0)$, but keep in mind that these need not be commutative. A Bruck loop in which the map $x\mapsto x+x=:2x$ is invertible is called a \emph{uniquely $2$-divisible Bruck loop}. Finite uniquely $2$-divisible Bruck loops have odd order \cite[Prop.~1]{Glauberman}. An identity these loops adhere to (cf. \cite[Lmm.~1]{Glauberman}) which we will employ in Section \ref{SubSubSec:BruckLpIsotopes} is
\begin{equation}\label{Eq:BruckLaw1}
    2(x+y)=x+(2y+x).
\end{equation}
If $(Q,*, \backslash_*, \rd)$ is a left symmetric Latin quandle with fixed element $e\in Q,$ then
\begin{equation}\label{Eq:BruckLoopIsotope}
    x+y=(x\rd e)*(ey)
\end{equation}
endows $Q$ with the structure of a uniquely $2$-divisible Bruck loop whose identity element is $e$. Conversely, given a uniquely $2$-divisible Bruck loop $(Q, +, e),$ the operation
\begin{equation}\label{Eq:LSLQIsotope}
    x* y=2x-y
\end{equation}
 is a left-symmetric Latin quandle multiplication. These two constructions are mutually inverse \cite[Thm.~1]{StuhlVoj}.

\begin{example}[\cite{KinyonNagy}]\label{Ex:Bpq}
Suppose $p>q$ are odd primes with $q\mid p^2-1.$ A unique nonassociative Bruck loop of order $pq$ --which we refer to as $B_{p, q}$-- exists and its multiplication may be defined in terms of the ring $\mathbb{Z}/_q \times \mathbb{Z}/_p$. First, select a primitive $q$th root of unity in a quadratic extension of $\mathbb{Z}/_p$; call it $\omega\in\mathbb{Z}/_p[\sqrt{t}]$. Then for $i\in \mathbb{Z}/_q$, set $\theta_i=\frac{2}{\omega^i+\omega^{-i}}\in \mathbb{Z}/_p$. Now we may define the Bruck loop operation by
\begin{equation}\label{Eq:Bpq}
    (i, j)\boxplus(k, l)=\left(i+k, \left(\frac{\theta_k+\theta_{i+k}}{\theta_k+\theta_k\theta_i}\right)j+\frac{\theta_{i+k}}{\theta_k}l\right).
\end{equation}
The identity element of $B_{p, q}$ is $(0, 0)$, and inverses coincide with those of the group $\mathbb{Z}/_q\times \mathbb{Z}/_p$.
\end{example}

\emph{Moufang loops} satisfy a stronger condition than \eqref{Eq:LBolLaw}:
\begin{equation}\label{Eq:MoufangLaw}
        x(y(xz))=((xy)x)z.
\end{equation}
Any Moufang loop is an IP loop. A loop is a \emph{commutative Moufang loop} (CML) if and only if it satisfies
\begin{equation}\label{Eq:Manin}
    x^2(yz)=(xy)(xz).
\end{equation}

Within the theory of commutative Moufang loops (CML's), those of exponent $3$ play a highly specialized role. To explain this specialized role, consider a CML $Q$, and its \emph{associator} $(x, y, z)=((xy)z)^{-1}(x(yz))$.  Then the \emph{derived subloop} $Q_1=\{(x, y, z)\mid x, y, z\in Q\}$ is the first component of the \emph{lower central series} $Q=Q_0,\dots, Q_i=\{(x, y, z)\mid x\in Q_{i-1}, y, z \in Q\}, \dots$.  The derived subloop has exponent $3$ \cite[Lemma 5.7]{Bruck58}.  The Bruck-Slaby theorem \cite[Theorem 10.1]{Bruck58} states that any CML on $n$ generators is nilpotent of class less than $n$.  In fact, the free CML on $n$ generators is of nilpotency class $n-1$, and for each $1\leq i\leq n-2$, the quotient $Q_{i}/Q_{i+1}$ is an elementary abelian group of exponent $3$ \cite{Sm78}.

\section{SYBE solutions from symmetric spaces: nondegenerate, braided, dihedral sets}\label{Sec:BDS}

\subsection{Basic Examples} \label{SubSec:Examples}
The purpose of this section is to introduce some basic examples of nondegenerate braided dihedral sets (BDS) and braided triality sets (BTS), with a focus on those which are not Latin. An obvious first step is to characterize the derived BDS and BTS.

\begin{proposition}\label{Prop:DerBTS/BDS}
Let $r:Q^2\to Q^2; (x, y)\mapsto (x\circ y, x)$ be a derived map. The pair $(Q, r)$ is
\begin{itemize}
    \item[(a)] a BDS if and only if $(Q, \circ)$ is a left-symmetric rack;
    \item[(b)] it is a BTS if and only if $(Q, \circ)$ is a distributive, totally symmetric quasigroup.
\end{itemize}

In particular, a derived BTS is necessarily Latin, and it has the form $$r(x, y)=(-x-y, x),$$ where $(Q, +)$ is a $\CMLtThm$.
\end{proposition}

\begin{proof}
    (a) Recall that a derived map is a solution if and only if $(Q, \circ)$ left distributive. Moreover, the derived formulation of \eqref{Eq:Di1} is trivial, while, for \eqref{Eq:Di2} it is precisely the left-symmetric identity $x\circ(x\circ y)=y$, which implies $(Q, \circ)$ is a left quasigroup. Thus, $(Q, \circ)$ is a left-symmetric rack if and only if $(\tau r)^2=1_{Q^{2}}$.

    (b) The derived version of \eqref{Eq:Tri1} is the semisymmetric identity $(x\circ y)\circ x=y$, and the derived version of \eqref{Eq:Tri2} is equivalent to commutativity $x\circ y=y\circ x$. A magma that is semisymmetric is necessarily a quasigroup, and the addition of commutativity forces total symmetry on $(Q, \circ)$. That in a totally symmetric, left distributive quasigroup $(Q, \circ)$, the operation $x\circ y$ must take the form $-x-y$ over some $\CMLt$ is a folkloric result from quasigroup theory. See, for instance, \cite[Prop. 2.1]{Donovan}.
\end{proof}

\begin{example}\label{Ex:Derived}
   We give some instances of derived solutions.
   \begin{itemize}
       \item[(a)] The trivial solution $\tau(x, y)=(y, x)$ is a derived BDS.
       \item[(b)] Taking addition over $\mathbb{Z}/_n$, BDS of the form $r(x, y)=(2x-y, x)$ correspond to the so-called dihedral quandles. These are BTS precisely when addition occurs in $(\mathbb{Z}/_3)^n$.
       \item[(c)] The smallest derived BDS which is not a quandle has order $2$. If $Q=\{0, 1\}$, then $L^\circ_0=L^\circ_1=(01)$, while $R^\circ_0$ is the constant function mapping onto $1$ and $R^\circ_1$ is the constant function mapping onto $0$.
       \item[(d)] The smallest derived BDS which is not a quandle and has non-constant right multiplication maps has order $4$. The multiplication table for $(Q, \circ)$ is given below.
          \begin{center}
    \begin{tabular}{c||c|c|c|c}
        $\circ$ & 0 & 1 & 2 & 3\\
        \hline
        \hline
        0 & 1 & 0 & 2 & 3\\
        \hline
        1 & 1  & 0 & 2 & 3 \\
        \hline
        2 & 0 & 1 & 3 & 2\\
        \hline
        3 &0 &1 & 3 & 2
    \end{tabular}
    \end{center}
   \end{itemize}
\end{example}

\begin{example}\label{Ex:LF-leftnotDerived}
Here is a nondegenerate, nonderived BDS in which $x\bullet y\neq x^S$. As we will see in Section \ref{Sec:LBDS}, this is impossible in the Latin case. Moreover, $(Q, \circ, \backslash_\circ)$ is an LF-left quasigroup, so we have a non-derived example that shows the necessity of the Latin assumption in Proposition \ref{Prop:F-qugpsiffSqmap}.\\
   \begin{center}
    \begin{tabular}{c||c|c|c}
        $\circ$ & 0 & 1 & 2\\
        \hline
        \hline
        0 & 1 & 0 & 2\\
        \hline
        1 & 1  & 0 & 2 \\
        \hline
        2 & 0 & 1 & 2
    \end{tabular}
    \phantom{SPACE}
        \begin{tabular}{c||c|c|c}
        $\bullet$ & 0 & 1 & 2\\
        \hline
        \hline
        0 & 1 & 1 & 0\\
        \hline
        1 & 0  & 0 & 1 \\
        \hline
        2 & 2 & 2 & 2
    \end{tabular}
    \end{center}
\end{example}

\begin{example}\label{Ex:ArbOrder}
 We now explain some more details behind a class of examples mentioned briefly in the introduction. It will follow that we can have BDS of arbitrary order. Let $n\geq 1$, and let $+$ refer to addition modulo $n$. Define $r:(\mathbb{Z}/_n)^2\to (\mathbb{Z}/_n)^2; (x, y)\mapsto (2x+y, -x)$. By induction, $r^k(x, y)=((k+1)x+ky, -kx-(k-1)y)$. It follows, then, that $|r|=n$. What is more, if we take $+$ to stand for addition on $\mathbb{Z}$, then $r$ has infinite order. Note also that $r$ is a Latin solution if and only if $n$ is odd.
\end{example}

\subsection{Solutions from symmetric spaces}\label{SubSec:SymmSpaces}
The concept of symmetric space offers a powerful, unifying framework for aspects of Lie groups and Riemannian manifolds. Loos's purely algebraic definition \cite[Sec~2.1.1]{LoosI} best suits our purposes, as it extends the concept to discrete spaces. 
\begin{definition}\label{Def:SymmSpace}
A magma $(Q,*)$ is a \emph{symmetric space} if it is 
a left symmetric quandle.
If $Q$ is a manifold, we require $*:Q^2\to Q$ to be a differentiable map, and that for all $x\in Q$, there is a neighborhood, $U$, of $x$ such that if $y\in U$, and $x* y=y,$ then $y=x.$
\end{definition}

Equivalently, a symmetric space $(Q, s)$ consists of a set $Q$ and a family of involutive \emph{symmetries around $x$} (for $x\in Q$): $s=\{s_x:Q\to Q\mid x\in Q\}$ such that each $s_x$ fixes $x$, and for every $x, y\in Q$ the symmetries around $x$ and $y$ adhere to the distributivity condition

\begin{equation}\label{Eq:SSD}
    s_{s_{x}(y)}(s_x(z))=s_x(s_y(z)).
\end{equation}
Moreover, if $Q$ is a manifold, we require each $s_x$ to be differentiable and that there be a neighborhood in which $x$ is the sole fixed point of $s_x$.  A \emph{pointed} symmetric space $(Q, s, o)$ is a symmetric space with a fixed \emph{base point} $o\in Q$.

\begin{proposition}\label{Prop:SymmSp}
Let $(Q, s, o)$ be a pointed symmetric space.  Then $(x, y)\mapsto (s_x(y), x)$ and $(x, y)\mapsto (s_x(s_o(y)), s_o(x))$ furnish nondegenerate, braided, dihedral sets on $Q$.
\end{proposition}

\begin{proof}
Begin by letting $r:(x, y)\mapsto(s_x(y), x)$. 
We have, $L^\circ_x=s_x$ and $R^\bullet_y=1_Q$. Hence, $(Q,r)$ is nondegenerate. Consider $(a,b,c)\in Q^3$. Then, by direct calculations in \eqref{Eq:SetBraidEqn} with one application of \eqref{Eq:SSD}, we obtain $(Q,r)$ is a braiding:
\begin{align*}
r^{12}r^{23}r^{12}(a,b,c)&=r^{12}r^{23}(s_a(b),a,c)=r^{12}(s_a(b),s_a(c),a)=(s_{s_a(b)}(s_a(c)),s_a(b),a)=\\
&=(s_a(s_b(c)),s_a(b),a)=r^{23}r^{12}r^{23}(a,b,c).
\end{align*}
And, finally, $(Q,r)$ is dihedral, again due to direct calculations and involutivity of symmetries. Namely, for $(a,b)\in Q^2$:
\begin{align*}
(\tau r)^2(a,b)=\tau r\tau(s_a(b),a)=\tau r(a,s_a(b))=\tau(s_a(s_a(b)),a)=(a,b).
\end{align*}
Now, we let $r:(x, y)\mapsto (s_x(s_o(y)), s_o(x))$, i.e. $L^\circ_x=s_x\circ s_o$ and $R^\bullet_y=s_o$.  Because $s_x$ and $s_o$ are invertible for all $x\in Q$, $(Q, r)$ is nondegenerate.  Consider $(a, b, c)\in Q^3$.  The third entry of both $X:=r^{12}r^{23}r^{12}(a, b, c)$ and $Y:=r^{23}r^{12}r^{23}(a, b, c)$ is $s^2_o(a)=a$.  The second entry of $X$ is $s_o(s_a(s_o(b)))$, while the second entry of $Y$ is $s_{s_o(a)}(b)=s_{s_o(a)}(s_o(s_o(b))=s_o(s_a(s_o(b)))$, by \eqref{Eq:SSD}.  Finally, note
\begin{align*}
    s_{s_a(s_o(b))}(s_o(s_{s_o(a)}(s_o(c))))&= s_{s_a(s_o(b))}(s_o(s_o(s_a(c))))\\
    &=s_{s_a(s_o(b))}(s_a(c))\\
    &=s_a(s_{s_o(b)}(c))\\
    &=s_a(s_{s_o(b)}(s_o(s_o(c))))\\
    &=s_a(s_o(s_b(s_o(c)))).
\end{align*}
The first expression is the first entry of $X$, while the last is the first entry of $Y$.  We are left to verify the dihedral condition.  Indeed, for $(a,b)\in Q^2$, $(\tau r)^2(a, b)=\tau r(s_o(a), s_a(s_o(b)))=(s^2_o(a), s_{s_o(a)}(s_o(s_a(s_o(b)))))=(a, s_o(s_a(s_a(s_o(b)))))=(a, b)$.
\end{proof}

\begin{example}(\cite[Example II.2.3]{LoosI})\label{Ex:SymmSpace}
    The following example has a similar flavor to the Latin solutions that are to come. Let $Q$ be the unit sphere in an inner-product space, $V$. For each $v\in Q$, one may define the Householder reflection in $v$ as the involution $H_v(x)=2\langle x, v\rangle v-x$ that reflects vectors across the hyperplane orthogonal to $v$. Fixing a base point $o\in Q$, gives BDS
    \begin{align*}
        (x, y)&\mapsto (H_x(y), x) \\
        (x, y)&\mapsto (H_x(H_o(y)), H_o(x)).
    \end{align*}
In fact, if we extend our scope to all $V$, we lose the topology of the symmetric space, but gain a linear BDS. That is, $(H_x\otimes 1_V)\tau$ and $(H_xH_o\otimes H_o)\tau$ are solutions of \eqref{Eq:SetBraidEqn} in the category of of vector spaces under the tensor product.
\end{example}

\section{Bruck loops and Latin braided dihedral sets}\label{Sec:LBDS}

\subsection{Latin braided dihedral sets}\label{SubSec:BDS}

\begin{lemma}\label{Lmm:DihedLmm1}
Let $(Q, r)$ be an LBDS.  Then all of the following hold:
\begin{itemize}
        \item[(a)] $x\bullet y=x^S$.
        \item[(b)]  $S^2=1_Q$.
        \item[(c)] $r$ is nondegenerate.
\end{itemize}
\end{lemma}

\begin{proof}
A proof of the fact that $x\bullet y=x^S$ may be found in Appendix \ref{appendix:human_proof}. It is based on output from \texttt{Prover9} \cite{Prover9}.

With $x\bullet y=x^S$ established, we note that condition (b) is equivalent to LBDS axiom \eqref{Eq:Di1}. Moreover, all right translation maps with respect to $(Q, \bullet)$ coincide with $S$. In other words, the right translation maps of $(Q, \bullet)$ are invertible, $(Q, \bullet)$ is a right quasigroup, and $r$ is nondegenerate.
\end{proof}

\begin{lemma}\label{Lmm:DihedLmm2}
    In any LBDS $(Q, r)$, all of the following hold:
    \begin{itemize}
        \item[(a)] $x^S\circ (x\circ y)=x\circ(x^S\circ y)=y$.
        \item[(b)] $x/^\circ y=y^S/^\circ x^S$
        \item[(c)] $(x\circ y)\circ z=x\circ(y\circ(x\circ z))$
    \end{itemize}
\end{lemma}

\begin{proof}
    (a) The fact that $x^S\circ (x\circ y)=y$ is the LBDS axiom $\eqref{Eq:Di2}$ recast in light of $x\bullet y=x^S$. Now, $x\circ(x^S\circ y)=y$ 
    follows by Lemma \ref{Lmm:IPMapInvo}.

  (b) Dividing $a^S\circ(a\circ b)=b$ on the right by $a\circ b$ gives $a^S=b/^\circ(a\circ b)$. Replace $a$ with $y/^\circ x$ and $b$ with $x$ in this statement, and use the fact that $S$ is an endomorphism (Lemma \ref{Lmm:SEndo}), and we have $x/^\circ y=x/^\circ((y/^\circ x)\circ x)=(y/^\circ x)^S=y^S/^\circ x^S$.

(c) This follows from the fact that $\circ$ adheres to \eqref{Eq:LFNew} and part (a): $x\circ(y\circ(x\circ z))=(x\circ y)\circ(x^S\circ(x\circ z))=(x\circ y)\circ z$.
\end{proof}

\noindent Thanks to Lemma \ref{Lmm:DihedLmm1}, LBDS is now understood to refer to nondegenerate, Latin solutions to the SYBE possessing the dihedral property.

We are now in a position to establish necessary and sufficient conditions on $(x\circ y, x\bullet y)$ to be an LBDS.

\begin{theorem}\label{Thm:LBDSClass}
The pair $(Q, r)$, with $r:(x, y)\mapsto (x\circ y, x\bullet y)$, constitutes an LBDS if and only if $x\bullet y=x^S$ and $(Q, \circ, \ld, /^\circ)$ is an LF-quasigroup in which
\begin{equation}\label{Eq:LBDSIP}
    x^S\circ (x\circ y)=y
\end{equation}
holds.
\end{theorem}

\begin{proof}
If we assume $(Q, r)$ is an LBDS, then $x\bullet y=x^S$ is Lemma \ref{Lmm:DihedLmm1}. Therefore, the LF identity is Proposition \ref{Prop:F-qgpiffBraid}. Finally, $\eqref{Eq:LBDSIP}$ follows from Lemma \ref{Lmm:DihedLmm2}(a).

Conversely, assume $r(x, y)=(x\circ y, x^S)$, where $(Q, \circ, \ld, /^\circ)$ is an LF-quasigroup in which $x^S\circ(x\circ y)=y$. We have to show that each of \eqref{Eq:YB1}-\eqref{Eq:Di2} hold. Well, by Proposition \ref{Prop:F-qgpiffBraid}, $r$ is a braiding, so \eqref{Eq:YB1}-\eqref{Eq:YB3} hold. Furthermore, \eqref{Eq:Di2} is equivalent to $x^S\circ(x\circ y)=y$. Finally, Lemma \ref{Lmm:IPMapInvo} tells us that in any left quasigroup in which $x^S\circ(x\circ y)=y$ holds, we have $S^2=1_Q$, and this proves \eqref{Eq:Di1}.
\end{proof}

\begin{corollary}\label{Cor:LBDS-quasigroup}
There is a one-to-one correspondence between LBDS and the subvariety of quasigroups generated by \eqref{Eq:LFNew} and the identity
\begin{align}
    (x\backslash_*x) *(x* y)&=y.\label{Eq:SqinLIP}
\end{align}
\end{corollary}

\begin{definition}\label{Def:LBDSQu}
 We refer to the variety of Corollary \ref{Cor:LBDS-quasigroup} as the \emph{variety of LBDS-quasigroups}.
\end{definition}

\begin{proposition}\label{Prop:DerRackProp1}
The derived rack $(Q, \rack)$ of an LBDS is a left symmetric, Latin quandle. Moreover, $(Q, \circ)$ is isotopic to this quandle.
\end{proposition}

\begin{proof}
We have to prove $x\rack x=x$ and $x\rack(x\rack y)=y$ for all $x, y\in Q$.  By Lemma \ref{Lmm:DihedLmm1}, we have $x\rack y=x\circ y^S$ as a reformulation of \eqref{Eq:StrRack}. This makes the isotopism obvious. Note $x\rack x=x\circ (x\ld x)=x$ by the quasigroup axioms.  What's more, $x\rack(x\rack y)=x\rack (x\circ y^S)=x\circ(x\circ y^S)^S=x\circ(x^S\circ y^{S^{2}})=x\circ(x^S\circ y)=y$ by Lemma \ref{Lmm:DihedLmm2}(a). Finally, to prove $\rack$ is a Latin quandle operation, we note that $x/^{\rack}y=x/^\circ y^S$ is the right division for this operation.
\end{proof}

\begin{remark}\label{Rmk:SymmSpaceDerRack}
    Moreover, in both cases, i.e. for solutions originating from a symmetric space $(Q,s)$ or a pointed symmetric space $(Q,s,o)$, the derived rack operation is defined as $x\rack y=s_x(y)$. Compare this to the situation described in Proposition \ref{Prop:DerRackProp1}.
\end{remark}

\begin{proposition}\label{Prop:LBDSIdemp}
    Let $(Q, \circ, \ld, /^\circ)$ be an LBDS-quasigroup. Then $Q$ contains an idempotent element. In particular, $\emph{Sq}_{/^{\circ}}(Q)$ is the set of idempotents, and upon this set, $\circ$ coincides with the derived rack operation.
\end{proposition}

\begin{proof}
    Let $P$ denote the set of idempotents of $(Q, \circ, \ld, /^\circ)$. If $x\in P$, then $x=x/^\circ x=\text{Sq}_{/^\circ}(x)\in \text{Sq}_{/^\circ}(Q)$ by Lemma \ref{Lmm:IdempQugp}.

    Now, we show $x/^\circ x\in P$ for all $x\in Q$. By Lemma \ref{Lmm:DihedLmm2}(c) $x/^\circ x=x^S/^\circ x^S=(x/^\circ x)^S=(x/^\circ x)\ld(x/^\circ x),$ so $x/^\circ x\in P$ by Lemma \ref{Lmm:IdempQugp}.

\end{proof}

\subsubsection{Bruck loop isotopes and the isomorphism problem for LBDS}\label{SubSubSec:BruckLpIsotopes}

\begin{lemma}\label{Lmm:IdempAuto}
    Let $e\in \emph{Sq}_{/^\circ}(Q)$ be idempotent in the LBDS quasigroup $(Q,\circ, \ld, /^\circ)$. Then $L^\circ_e:x\mapsto e\circ x$ is an automorphism of $(Q, \circ, \ld, /^\circ)$.
\end{lemma}

\begin{proof}
    This is an immediate consequence of the LF identity \eqref{Eq:LFNew}.
\end{proof}

Recall the correspondence between left symmetric Latin quandles and uniquely 2-divisible Bruck loops established by \eqref{Eq:BruckLoopIsotope} and \eqref{Eq:LSLQIsotope}. This correspondence gives us the following result.

\begin{lemma}\label{Lmm:BruckLoopFromLBDS}
    Let $(Q, \circ, \ld, /^\circ)$ be an LBDS quasigroup with left divisional squaring map $S:x\mapsto x\ld x$. Fix an idempotent element $e\in \emph{Sq}_{/^\circ}(Q)$. Then the operation
    \begin{equation}
        x+y=(x/^\circ e)\circ(e\circ y)
    \end{equation}
   endows $Q$ with the structure of a uniquely 2-divisible Bruck loop $(Q, +, e)$ in which $-x=e\circ x^S$, $2x=x\circ e$, and $S$ is an automorphism. Moreover, the isomorphism class of $(Q, +, e)$ is independent of our choice of idempotent.
\end{lemma}
\begin{proof}
   Since $(Q, \rack)$ is a left symmetric Latin quandle, if we can show that $x+y$ is the isotope of \eqref{Eq:BruckLoopIsotope} associated with $(Q, \rack),$ then we are done. Indeed, $x+y=(x/^\circ e)\circ(e\circ y)=(x/^\circ e)\circ(e\circ y^S)^S=(x/^{\rack} e)\circ(e\rack y)^S=(x/^{\rack}e)\rack(e\rack y)$. Obtaining $-x=e\circ x^S$ is immediate from the quasigroup axioms. That $2x=x\circ e$ is due to the LF identity:
   \begin{align*}
       2x&=(x/^\circ e)\circ(e\circ x)\\
       &=((x/^\circ e)\circ e)\circ ((x/^\circ e)^S\circ x)\\
       &=x\circ ((x/^\circ e)\ld x)\\
       &=x\circ e.
   \end{align*}
That $S$ is an automorphism of $(Q, +, e)$ follows from the fact that $S$ commutes with $\circ$ and $/^\circ$, along with the fact that $e^S=e.$

Finally, we show the independence of the choice $e\in \text{Sq}_{/^\circ}(Q)$. Indeed, suppose that $f\in \text{Sq}_{/^\circ}(Q)$ and define $x\boxplus y:=(x/^\circ f)\circ(f\circ y)$. Let $g:=f/^\circ e$ so that $g\circ e=f.$  We claim $L^\circ_g$ is an isomorphism of $(Q, +, e)$ onto $(Q, \boxplus, f)$. Since $f$ and $e$ are idempotent and $S$ is an automorphism of $(Q, \circ, \ld, /^\circ)$, $g^S=g$, $g$ is idempotent, and, by Lemma \ref{Lmm:IdempAuto}, $L^\circ_g$ is an automorphism of $(Q, \circ, \ld, /^\circ)$. Thus, $L^\circ_g(x+y)=((g\circ x)/^\circ(g\circ e))\circ((g\circ y)\circ(g\circ e))=(L^\circ_g(x)/^\circ f)\circ(L^\circ_g(y)\circ f)=L^\circ_g(x)\boxplus L^\circ_g(y)$.
\end{proof}

\begin{proposition}\label{Prop:OddOrderLBDS}
    If $(Q, r)$ is a finite LBDS, then $|Q|$ is odd.
\end{proposition}

\begin{proof}
    Since $(Q, \circ, \ld, /^\circ)$ is isotopic to a uniquely 2-divisible Bruck loop, $|Q|$ is odd (cf. \cite[Prop.~1]{Glauberman}).
\end{proof}

\begin{lemma}\label{Lmm:LBDSFromBruck}
    Let $(Q, +, e)$ be a uniquely $2$-divisible Bruck loop with involutive automorphism $S$. Then the operations $x\circ y=2x-y^S$, $x\ld y=2x^S-y^S$, $x/^\circ y=y^S+\frac{1}{2}(-y^S+x)$ make $(Q, \circ, \ld, /^\circ)$ an LBDS-quasigroup and in this quasigroup, $x^S=x\ld x$.
\end{lemma}

\begin{proof}
    We need to check that $(Q, \circ, \ld,/^\circ)$ is a quasigroup satisfying identities \eqref{Eq:LFNew} and \eqref{Eq:SqinLIP}.  Identities (IL), (SL), and \eqref{Eq:SqinLIP} follow from the LIP and AIP. The fact that $x^S=x\ld x$ is due to power-associativity of Bruck loops. Furthermore, \eqref{Eq:LFNew} holds because
\begin{align*}
(x\circ y)\circ (x^S\circ z)&=2(2x-y^S)-(2x^S-z^S)^S
=2(2x-y^S)-(2x-z)\\
&=(2x+(-2y^S+2x))-(2x-z)
=2x+(-2y^S+(2x+(-2x+z)))\\
&=2x+(-2y^S+z)
=2x-(2y-z^S)^S
=x\circ (y\circ z).
\end{align*}
We used \eqref{Eq:BruckLaw1}, \eqref{Eq:LBolLaw}, LIP and AIP in that order. The right division identity (IR) follows from \ref{Eq:BruckLaw1} and the LIP: $(x\circ y)/^\circ y=(2x-y^S)/^\circ y=y^S+\frac{1}{2}(-y^S+(2x-y^S))=y^S+\frac{1}{2}(2(-y^S+x))=x.$ To prove (SR), we first note that by \eqref{Eq:BruckLaw1}, $$2(y^S+\frac{1}{2}(-y^S+x))=y^S+((-y^S+x)+y^S).$$ Therefore,
\begin{align*}
    (x/^\circ y)\circ y&=2(y^S+\frac{1}{2}(-y^S+x))-y^S
    =(y^S+((-y^S+x)+y^S))-y^S\\
    &=y^S+((-y^S+x)+(y^S-y^S))
    =x.
\end{align*}
The penultimate equality is by \eqref{Eq:LBolLaw}, while the final one comes from LIP.
\end{proof}

\noindent Lemmas \ref{Lmm:BruckLoopFromLBDS} and \ref{Lmm:LBDSFromBruck} lead directly to the following.

\begin{theorem}\label{Thm:LBDSCharacterization}
    The pair $(Q, r)$ with $r(x, y)=(x\circ y, x\bullet y)$ constitutes an LBDS if and only if there is a uniquely $2$-divisible Bruck loop $(Q, +, e)$ and an involution $S\in \emph{Aut}(Q, +, e)$ such that $r(x, y)=(2x-y^S, x^S)$.
\end{theorem}

\begin{corollary}\label{Cor:ExponentBruckLoop}
    Let $(Q, r)$ be an LBDS such that, as a permutation of $Q^2$, $|r|=n$. If $(Q, +, e)$ is the associated Bruck loop, then its exponent divides $n$. In particular, if $n$ is prime, then $(Q, +, e)$ is a Bruck loop of exponent $n.$
\end{corollary}

\begin{proof}
    An easy induction proof, bearing in mind that Bruck loops are power-associative, shows that whenever $x\in Q$, we have $(e, x)=r^n(e, x)=(-nx^S, (-n+1)x)$.
\end{proof}

\begin{lemma}\label{Lmm:BruckandBack}
    The constructions of Lemmas \ref{Lmm:BruckLoopFromLBDS} and \ref{Lmm:LBDSFromBruck} are mutually inverse. That is, if $(Q, +, e)$ is a uniquely 2-divisible Bruck loop with involutive automorphism $S$ and $x\circ y:=2x-y^S,$ then $x+y=(x/^\circ e)\circ(e\circ y)$. Conversely, if $(Q, \circ, \ld, /^\circ)$ is an LBDS-quasigroup with idempotent $e$, $x^S:=x\ld x$, and $x+y:=(x/^\circ e)\circ(e\circ y)$ is the associated Bruck loop operation, then $x\circ y=2x-y^S$.
\end{lemma}

\begin{proof}
    Suppose $(Q, +, e)$ is a uniquely $2$-divisible Bruck loop with involutive automorphism $S$, so that $x\circ y:=2x-y^S$ and $x/^\circ y=\frac{1}{2}(y^S+((-y^S+x)+y^S))$. Then $x/^\circ e=\frac{1}{2}x$ since $-e^S=e^S=e.$ Similarly, $e\circ y=-y^S.$ Thus, $(x/^\circ e)\circ (e\circ y)=(\frac{1}{2}x)\circ(-y^S)=2(\frac{1}{2}x)-(-y^S)^S=x+y$.

    Now, begin by assuming that $(Q, \circ, \ld, /^\circ)$ is an LBDS-quasigroup with idempotent $e$. Note $-y^S=(e\circ y^S)^S=e\circ y$. Thus,
    \begin{align*}
        2x-y^S&=(x\circ e)+(e\circ y)\\
        &=((x\circ e)/^\circ e)\circ (e\circ (e\circ y))\\
        &=x\circ (e^S\circ(e\circ y))\\
        &=x\circ y.
    \end{align*}
\end{proof}

\begin{remark}\label{Rmk:LBDSLang}
  Theorem \ref{Thm:LBDSCharacterization} and Lemma \ref{Lmm:BruckandBack} allow us to describe LBDS entirely in the language of uniquely $2$-divisible Bruck loops. When we say $(Q, +, e, S)$ is an LBDS, it will be understood that $(Q, +, e)$ is a uniquely $2$-divisible Bruck loop and $S\in \text{Aut}(Q, +, e)$ has order dividing $2$. Moreover, when we say $(Q, \circ, \ld, /^\circ)$ is the associated LBDS-quasigroup, it will be understood that $x\circ y=2x-y^S$, $x\ld y=2x^S-y^S$, $x/^\circ y=y^S+\frac{1}{2}(-y^S+x)$ and that $x+y=(x/^\circ e)\circ(e\circ y)$.
\end{remark}

\begin{lemma}\label{Lmm:HomCorrespondence}
    Let $(Q, +, e, S)$ and $(P, \boxplus, f, T)$ be LBDS with corresponding quasigroups $(Q, \circ, \ld ,/^\circ)$ and $(P, *, \backslash_*, /^*)$.
    \begin{itemize}
        \item[(a)] If $\varphi:Q\to P$ is a Bruck loop homomorphism such that $\varphi S=T\varphi$, it is an LBDS-quasigroup homomorphism.
        \item[(b)] If $\psi:Q\to P$ is an LBDS-quasigroup homomorphism and $g=f/^*\psi(e)$, then $\varphi:=L^*_g\psi$ is a Bruck loop homomorphism with $\varphi S=T\varphi$.
        \end{itemize}
\end{lemma}

\begin{proof}
    (a) For all $x, y\in Q$, $\varphi(x\circ y)=\varphi(2x-y^S)=\varphi(2x)\boxminus \varphi(y^S)=2\varphi(x)\boxminus \varphi(y)^T=\varphi(x)*\varphi(y).$\\

    (b) By assumption, $\psi(e), f\in P$ are idempotent, and since $T$ is an automorphism of $(P, *, \backslash_*, /^*)$, we can conclude $g=f/^*\psi(e)$ is idempotent. Thus, for all $x\in Q,$ $\varphi(x^S)=g*\psi(x^S)=g*\psi(x\ld x)=g*(\psi(x)\backslash_* \psi(x))=g*(\psi(x))^T=(g*\psi(x))^T=\varphi(x)^T$. That is, $\varphi S=T\varphi.$ To see that $\varphi$ is a loop homomorphism, consult the last two sentences of the proof of Lemma \ref{Lmm:BruckLoopFromLBDS}; the argument is nearly identical.
\end{proof}

\noindent The following theorem is an immediate consequence of Lemma \ref{Lmm:HomCorrespondence}

\begin{theorem}\label{Thm:LBDSIsom}
    Let $(Q, +, e, S)$ and $(P, \boxplus, f, T)$ be LBDS. The associated LBDS quasigroups $x\circ y:=2x-y^S$ and $x*y=2x\boxminus y^T$ are isomorphic if and only if there is a Bruck loop isomorphism $\varphi:Q\to P$ such that $\varphi S=T\varphi.$
\end{theorem}

\subsubsection{Squaring the right division, substructures and extension}\label{SubSubSec:SquareLBDS}
Suppose that $(Q, +, e, S)$ is an LBDS with associated LBDS-quasigroup $(Q, \circ, \ld, /^\circ)$. Recall from Proposition \ref{Prop:LBDSIdemp} that the idempotents of $(Q, \circ, \ld, /^\circ)$ coincide with $\text{Sq}_{/^\circ}(Q)$. By Lemma \ref{Lmm:IdempQugp}, $x\in \text{Sq}_{/^\circ}(Q)$ if and only if $x^S=x$. What's more, if $x, y\in \text{Sq}_{/^\circ}(Q)$, then $(x\circ y)^S=x^S\circ y^S=x\circ y$, so $x\circ y\in \text{Sq}_{/^\circ}(Q)$. That is, $\text{Sq}_{/^\circ}(Q)$ is a subquasigroup of $Q$ and in this set, the LBDS-operation takes the form $$x\circ y=2x-y.$$
We can now also plainly see that $\text{Sq}_{/^\circ}(Q)$ is also a subloop of $(Q, +, e)$, as $(x+y)^S=x^S+y^S=x+y$ for all $x, y\in Q$.
\noindent Next, define

\begin{equation}\label{Eq:Qe}
    Q_e=\{x\in Q\mid \text{Sq}_{/^\circ}(x)=e\}
\end{equation}
to be the pre-image of the identity with respect to $\text{Sq}_{/^\circ}$. Combining the LIP and the AIP, we can see that $\text{Sq}_{/^\circ}(x)=x^S+\frac{1}{2}(-x^S+x)=e$ if and only if $x=-x^S$. Therefore, whenever $x, y\in Q_e$, $(x\circ y)^S=x^S\circ y^S=(-x)\circ(-y)=(e\circ x)\circ (e\circ y)=e\circ(x\circ y)=-(x\circ y)$, and so $x\circ y\in Q_e$. In other words, $Q_e$ is a subquasigroup of $Q$ in which $$x\circ y=2x+y.$$ Moreover, $Q_e$ is a subloop, as $(x+y)^S=x^S+y^S=-x-y=-(x+y)$ for all $x, y\in Q_e$.

\noindent We summarize this discussion via the following proposition.

\begin{proposition}\label{Prop:IdempSubqugp}
    Let $(Q, +, e, S)$ be an LBDS with associated quasigroup $(Q, \circ, \ld, /^\circ)$.
    \begin{itemize}
        \item[(a)] The set of idempotents of $(Q, \circ, \ld, /^\circ)$, $\emph{Sq}_{/^\circ}(Q)$, forms a subquasigroup and a subloop of $(Q, +, e)$ that coincides with the set $$\{x\in Q\mid x^S=x\}$$ of fixed points of $S$.
        \item[(b)] The set $Q_e$ of \eqref{Eq:Qe} forms a subquasigroup of $(Q, \circ, \ld, /^\circ)$ and a subloop of $(Q, +, e)$ that coincides with $$\{x\in Q\mid x^S=-x\}.$$
    \end{itemize}
\end{proposition}

\begin{remark}\label{Rmk:SqMapEndomorphism}
Now we have shown that the image and, kernel (so to speak) of $\text{Sq}$ are substructures of the $(Q, +, e, S)$, it begs the question, is $\text{Sq}_{/^\circ}$ always an endomorphism? As we will see in Subsections \ref{SubSec:SmallOrderLBDS} (Remark \ref{Rmk:3pSqMapAnalysis}) and \ref{SubSec:SmallOrderLBTS} (Lemma \ref{Lmm:L3LBTS} and its proof), the answer is no.
\end{remark}
We will devote the rest of the section to studying the nature of $(Q, +, e, S)$ when $\text{Sq}_{/^\circ}$ is an endomorphism.
\begin{lemma}\label{Lmm:LBDSEndo}
  Let $(Q, +, e, S)$ be an LBDS. Suppose $\varphi:Q\to Q$ is a function such that $\varphi(e)=e$ and $\varphi S=S\varphi$. Then $\varphi$ is an endomomorphism of $(Q, +, e)$ if and only if it is an endomorphism of $(Q, \circ, \ld, /^\circ)$. In particular, $\emph{Sq}_{/^\circ}$ is an endomorphism of $(Q, +, e)$ if and only if it is an endomorphism of $(Q, \circ, \ld, /^\circ).$
\end{lemma}

\begin{proof}
    ($\implies$) This is just a special case of Lemma \ref{Lmm:HomCorrespondence}(a).

    $(\impliedby)$ By Lemma \ref{Lmm:HomCorrespondence}(b), $\psi=L^\circ_e \varphi$ is a loop endomorphism. Moreover, $L_e^\circ(x)=-x^S$ for all $x\in Q$, so $L^\circ_e$ is a loop automorphism. Therefore, $\varphi=(L^\circ_e)^{-1}\psi$ is a loop endomorphsim.

    The last statement of the lemma follows since $S$ is an automorphism with respect to $/^\circ$, and because $\text{Sq}_{/^\circ}(e)=e^S+\frac{1}{2}(-e^S+e)=e.$
    \end{proof}

\begin{theorem}\label{Thm:LBDSSplitExt}
Let $(Q, +, e, S)$ be an LBDS with associated quasigroup $(Q, \circ, \ld, /^\circ)$. If $\emph{Sq}_{/^\circ}$ is an endomorphism of $(Q, +, e)$, then we have a split extension

\begin{equation}\label{Eq:LBDSSPlitExt}
    \{e\}\to Q_e\to Q\to \emph{Sq}_{/^\circ}(Q)\to \{e\}
\end{equation}
in both the category of LBDS-quasigroups and the category of Bruck loops.
\end{theorem}

\begin{proof}
By assumption, the first isomorphism theorem for loops applies (cf. \cite[Thm.~IV.1.1]{Bruck58}). That is, $\text{Sq}_{/^\circ}(Q)\cong Q/Q_e$ as loops. The inclusion of $\text{Sq}_{/^\circ}(Q)$ in $Q$ splits \eqref{Eq:LBDSSPlitExt} if $Q_e\cap \text{Sq}_{/^\circ}(Q)=\{e\}$. Indeed, by Proposition \ref{Prop:IdempSubqugp}, if $x\in Q_e\cap \text{Sq}_{/^\circ}(Q)$ then $x=x^S=-x$, so $2x=e$, which, by $2$-divisibility, yields $x=e$. By Lemma \ref{Lmm:LBDSEndo}, the sequence \eqref{Eq:LBDSSPlitExt} also lives in the category of LBDS-quasigroups.
\end{proof}

\begin{corollary}\label{Cor:LBDSSPlitExt}
    Let $(Q, +, e, S)$ be an LBDS such that $\emph{Sq}_{/^\circ}$ is endomorphic,
    \begin{itemize}
        \item[(a)] the Bruck loop automorphism $S$ is upper-triangularizable in the sense that $Q$ is a split extension of a normal subloop upon which $S$ acts as $-1$ with a subloop upon which $S$ acts as $1.$
        \item[(b)] $(Q, +, e, S)$ is a split extension of a solution of type \eqref{Eq:GenofSmSoln} with one of type \eqref{Eq:GenofDerSoln}.
    \end{itemize}
\end{corollary}

\subsection{LBDS of orders $p,$ $p^2$, $3p$, $27$, and $81$}\label{SubSec:SmallOrderLBDS}
We will now use Theorem \ref{Thm:LBDSIsom} to classify LBDS of orders $p, p^2$ and $3p$, where $p$ is an odd prime. Then, with the help of the \texttt{LOOPS} package in \texttt{GAP}, we shall enumerate isomorphism classes for LBDS of order $27$ and $81$.

\subsubsection{Orders $p$ and $p^2$: the associative case} \label{SubSubSec:Orderpp2} Bruck loops of order $p$ and $p^2$ --for $p$ and odd prime-- are groups \cite{Burn}. Therefore, the description of LBDS at these orders is straightforward. Indeed, if $(Q, +, e, S)$ has order $p$, then the underlying Bruck loop must be $(\mathbb{Z}/_p, +, 0)$, and $S=\pm 1\in(\mathbb{Z}/_p)^*$.
Similarly, if $(Q, r)$ has order $p^2$ and the underlying Bruck loop is $(\mathbb{Z}/_{p^2}, +, 0),$ then $S=\pm 1\in (\mathbb{Z}/_{p^2})^*$. Next, suppose $|Q|=p^2$ and the underlying loop is $((\mathbb{Z}/_{p})^2, +, 0)$. The possible rational canonical forms for involutions of $(\mathbb{Z}/_p)^2$ are $$\left(\begin{array}{cc}
    1 & 0 \\
    0 & 1
\end{array}\right), \left(\begin{array}{cc}
    -1 & 0 \\
    0 & -1
\end{array}\right), \text{ and} \left(\begin{array}{cc}
    1 & 0 \\
    0 & -1
\end{array}\right),$$ so $S$ is conjugate to one of these. We summarize our discussion in the following theorem.

\begin{theorem}\label{Thm:LBDSpp2}
    Let $p$ be an odd prime. There are $2$ isomorphism classes of LBDS of order $p$ and $5$ isomorphism classes of order $p^2$. These classes have the following representatives:
    \begin{itemize}
        \item[(a)] $(\mathbb{Z}/_p, +, 0, \pm 1)$;
        \item[(b)] $(\mathbb{Z}/_{p^2}, +, 0, \pm 1)$, $((\mathbb{Z}/_p)^2, 0, \pm 1)$, and $((\mathbb{Z}/_p)^2, +, 0, 1\oplus -1)$.
    \end{itemize}
\end{theorem}

\subsubsection{Order $3p$}\label{SubSubSec:Order3p}
Now, we assume $(Q, +, e, S)$ has order $3p$, where $p>3$ is prime. If $(Q, +, e)$ is associative, it is an abelian group and thus $(Q, +, e)\cong(\mathbb{Z}/_{3p}, +, 0)$. So once, again, $S=\pm1$.

To address the nonassociative case, we proceed with a general discussion of $B_{p, 3}$. When $q=3$, the operation of \eqref{Eq:Bpq} can be summarized by the following condensed multiplication table

\begin{center}
\begin{tabular}{c||c|c|c}
$B_{p, 3}$ & $(0, l)$ & $(1, l)$ & $(-1, l)$\\
\hline
\hline
    $(0, j)$ & $(0, j+l)$& $(1, j+l)$ & $(-1, j+l)$   \\
    \hline
   $(1, j)$  & $(1, j-2l)$  & $(-1, -2j+l)$ & $(0, 2^{-1}(-j-l))$ \\
   \hline
     $(-1, j)$  & $(-1, j-2l)$  & $(0, 2^{-1}(-j-l))$ & $(1, -2j+l)$
\end{tabular}
\end{center}

\noindent In the ensuing discussion, whenever possible, we state our results in terms of $B_{p, q}$. However, ultimately, we will only be describing LBDS over $B_{p, 3}$.

\begin{lemma}\label{Lmm:BpqGens}
    The standard generators $(1, 0), (0, 1)\in \mathbb{Z}/_q\times\mathbb{Z}/_p$ generate $B_{p, q}$. In particular, for $(m, n)\in B_{p, q},$ $$(m, n)=((0, 1)^{\boxplus n})\boxplus((1, 0)^{\boxplus m}).$$
\end{lemma}

\begin{proof}
    Clearly, $(i, 0)\boxplus(k, 0)=(i+k, 0)$ for any $i, k\in \mathbb
Z/_{q}$. Therefore, $(1, 0)^{\boxplus m}=(m , 0)$. Since $\theta_0=1$, $(0, j)\boxplus (0, l)=(0, j+l)$ for all $j, l\in \mathbb{Z}/_p$. Thus, $(0, 1)^{\boxplus n}=(0, n).$ Conclude
    \begin{align*}
        ((0, 1)^{\boxplus n})\boxplus((1, 0)^{\boxplus m})&=(0, n)\boxplus(m, 0)\\
        &=\left(0+m, \frac{2\theta_m}{2\theta_m}n+\frac{\theta_m}{\theta_m}0\right)\\
        &=(m, n).
    \end{align*}
\end{proof}

\begin{lemma}\label{Lmm:BpqSublps}
    The Bruck loop $B_{p, q}$ contains a unique subloop of order $p$, $\langle(0, 1)\rangle$, and $p$-many subloops of order $q$, $\langle(1, n)\rangle$ for all $n\in \mathbb{Z}/_p.$
\end{lemma}

\begin{proof}
    This is due to \cite[Thm.~1]{NR}, which states that any Bol loop of order $pq$ has a unique subloop of order $p$ and, if the loop is nonassociative, each non-identity element has order $p$ or order $q$.
\end{proof}

\begin{definition}\label{Def:LambdaGroup}
   For an odd prime $p$, let $\Lambda_p$ denote the set of functions $\lambda_b:\mathbb{Z}/_3\to \mathbb{Z}/_p$ such that $\lambda_b(0)=0$, $\lambda_b(\pm 1\mod 3)=\pm b\mod p$ for some $b\in\mathbb{Z}/_p$. Moreover, for $a\in \mathbb{Z}/_3$, let $_a\lambda_b:x\mapsto \lambda_b(ax)$.
\end{definition}

\begin{lemma}\label{Lmm:LPSubSP}
    The set $\Lambda_p$ forms a $1$-dimensional subspace of the $\mathbb{Z}/_p$-vector space $\emph{Hom}_{\emph{Set}}(\mathbb{Z}/_3, \mathbb{Z}/_p)$. Moreover, for $a\in \mathbb{Z}/_3,$ $_a\lambda_b\in\Lambda_p$.
\end{lemma}

\begin{proof}
   For $b, c, d\in \mathbb{Z}/_p$, $(d\lambda_c+\lambda_b)(0)=d0+0=0=\lambda_{dc+b}(0)$ and $(d\lambda_c+\lambda_b)(\pm 1 \mod 3)=d(\pm c)+(\pm b)\mod p=\pm(dc+b)\mod p=\lambda_{dc+b}(\pm1 \mod 3)$. Clearly, $\lambda_c=\lambda_b$ if and only if $c=b$, so $\dim(\Lambda_p)=1$. Finally, note that $_0\lambda_b=\lambda_0$, $_1\lambda_b=\lambda_b$, and $_{-1}\lambda_b=\lambda_{-b}$.
\end{proof}

\begin{remark}\label{Rmk:ModClasses}
    The proof of Lemma \ref{Lmm:LPSubSP} shows that when performing calculations, we can interchange $_a\lambda_b$ and $\lambda_{ab}$, where it is understood that the only representatives we'll use for $a$ are $0$ and $\pm 1$, both modulo $3$ and $p$.
\end{remark}

\begin{proposition}\label{Prop:Ap3Group}
    For $a\in \mathbb{Z}/_3$, $b\in \mathbb{Z}/_{p}$, and $\lambda_c\in \Lambda_p$, define a function $$[a, b, \lambda_c]:\mathbb{Z}/_3\times \mathbb{Z}/_p\to\mathbb{Z}/_3\times \mathbb{Z}/_p;(x, y)\mapsto (ax, by+\lambda_c(x)).$$ The set
    \begin{equation}
        \mathcal{A}_{p, 3}=\{[a, b, \lambda_c]\mid a\in (\mathbb{Z}/_3)^*, b\in (\mathbb{Z}/_p)^*, \lambda_c\in \Lambda_p\}
    \end{equation}
    forms an order $2p(p-1)$ group under composition of functions.
\end{proposition}

\begin{proof}
    The identity permutation is realized by $[1, 1, \lambda_0]$. If $\varphi=[a, b, \lambda_c]$ and $\psi=[d, e, \lambda_f]$, then $\psi\varphi:(x, y)\mapsto (dax, eby+e\lambda_c(x)\phantom{.}+\phantom{.}_a\lambda_f(x))$. So by Lemma \ref{Lmm:LPSubSP}, $e\lambda_c+ \phantom{.}_a\lambda_{f}=\lambda_{ec}+\phantom{.}_a\lambda_f\in \Lambda_p$ and $\psi\varphi=[da, eb, \lambda_{ec}+\phantom{.}_a\lambda_f] \in\mathcal{A}_{p, 3}$. The inverse of $[a, b, \lambda_c]$ is $[a^{-1}, b^{-1},\phantom{.}_{a^{-1}} \lambda_{-b^{-1}c}]$.

    Establishing the order of $\mathcal{A}_{p, 3}$ is a matter of showing $[*]:(a, b, \lambda_c)\mapsto [a, b, \lambda_c]$ is injective. If $\varphi=[a, b, \lambda_c]=[d, e, \lambda_f]=\psi,$ then $(a, c)=\varphi(1, 0)=\psi(1, 0)=(d, f)$. Moreover, $(0, b)=\varphi(0, 1)=\psi(0, 1)=(0, e)$. This proves $(a, b, \lambda_c)=(d, e, \lambda_f)$, so $[*]$ does inject.
\end{proof}

\begin{theorem}\label{Thm:Bp3Aut}
    Let $B_{p, 3}$ be the nonassociative Bruck loop of order $3p$. Then $\emph{Aut}(B_{p, 3})=\mathcal{A}_{p, 3}$.
\end{theorem}

\begin{proof}
First, we show that all automorphisms of $B_{p, 3}$ are of the form $[a, b, \lambda_c]$. Indeed, let $\alpha\in \text{Aut}(B_{p, 3})$. By Lemma \ref{Lmm:BpqSublps}, $\alpha(0, 1)=(0, b)$ for some $b\in(\mathbb{Z}/_p)^*$. Moreover, $\alpha(1, 0)=(a, c)$ for some $a\in(\mathbb{Z}/_3)^*$ and $c\in \mathbb{Z}/_p$. By Lemma \ref{Lmm:BpqGens}, for $(m, n)\in \mathbb{Z}/_3\times \mathbb{Z}/_p$ we have $\alpha(m, n)=\alpha(0, 1)^{\boxplus n}\boxplus\alpha(0, 1)^{\boxplus m}=(0, bn)\boxplus(a, c)^{\boxplus m}$. Now, if $m=0,$ $(0, bn)\boxplus(a, c)^{\boxplus m}=(0, bn)=(am, bn+\lambda_c(m))$. If $m=1$, $(0, bn)\boxplus(a, c)^{\boxplus m}=(a,bn+ c)=(am, bn+\lambda_c(m))$. If $m=-1$, $(0, bn)\boxplus(a, c)^{\boxplus m}=(-a, bn-c)=(am, bn+\lambda_c(m)).$ Either way, $\alpha(m, n)=[a, b, \lambda_c](m, n)$.

Now, we show each $\alpha=[a, b, \lambda_c]\in \text{Aut}(B_{p, 3})$. Fix $(i, j), (k, l)\in \mathbb{Z}/_3\times \mathbb{Z}/_p$. Then
\begin{equation}\label{Eq:AutEq1}
   \alpha((i, j)\boxplus(k, l))=\left(a(i+k), \frac{\theta_k+\theta_{i+k}}{\theta_k+\theta_k\theta_i}bj+\frac{\theta_{i+k}}{\theta_k}bl+\lambda_c(i+k)\right),
\end{equation}
while
\begin{equation}\label{Eq:AutEq2}
    \alpha(i, j)\boxplus\alpha(k, l)=\left(ai+ak, \frac{\theta_{ak}+\theta_{a(i+k)}}{\theta_{ak}+\theta_{ak}\theta_{ai}}(bj+\lambda_c(i))+\frac{\theta_{a(i+k)}}{\theta_{ak}}(bl+\lambda_c(k))\right).
\end{equation}
Note that $\theta_{ax}=\theta_x$ for all $a, x\in \mathbb{Z}/_3$. Therefore, equality of \eqref{Eq:AutEq1} and \eqref{Eq:AutEq2} comes down to whether or not
\begin{equation}\label{Eq:AutEq3}
    \lambda_c(i+k)=\frac{\theta_{k}+\theta_{i+k}}{\theta_k+\theta_k\theta_i}\lambda_c(i)+\frac{\theta_{i+k}}{\theta_k}\lambda_c(k).
\end{equation}
Using the condensed multiplication table for $B_{p, 3}$ above, one may check the validity of \eqref{Eq:AutEq3} for each possible $i$-$k$ pairing. For instance, if $i=k=1$, then $\frac{\theta_1+\theta_{-1}}{\theta_1+\theta_1\theta_1}\lambda_c(1)+\frac{\theta_{-1}}{\theta_1}\lambda_c(1)=-2\lambda_c(1)+\lambda_c(1)=-2c+c=-c=\lambda_c(-1)=\lambda_c(1+1)$.
\end{proof}

\begin{proposition}\label{Prop:Bp3Invol}
    The involutions in $\emph{Aut}(B_{p, 3})$ belong to four conjugacy classes. They are
    \begin{enumerate}
        \item the singleton containing the identity, $\{[1, 1, \lambda_0]\}$,
        \item the singleton containing inversion with respect to \eqref{Eq:Bpq}, $\{[-1, -1, \lambda_0]\}$,
        \item $\{[1, -1, \lambda_c]: c\in \mathbb{Z}/_p\}$, and
        \item $\{[-1, 1, \lambda_c]: c\in \mathbb{Z}/_p\}$.
    \end{enumerate}
\end{proposition}

\begin{proof}
    First, we show that (1)-(4) account for all involutions. Suppose $[a, b, \lambda_c]=[a, b, \lambda_c]^{-1}=[a^{-1}, b^{-1}, \phantom{.}_{a^{-1}}\lambda_{-b^{-1}c}]$. Then $a=\pm 1\mod 3$ and $b=\pm1\mod p$. Assuming either $a=1\mod 3$ and $b=1\mod p$ or $a=-1\mod 3$ and $b=-1\mod p$ both lead to $\lambda_c=\lambda_{-c}$, which forces $c=0\mod p$. This accounts for (1) and (2). Now, we confirm that any choice of $c\in \mathbb{Z}/_p$ makes $[1, -1, \lambda_c]$ and $[-1, 1, \lambda_c]$ involutions. Note
    \begin{equation*}
        \xymatrix{
        (x, y)\ar@{|->}[r]^-{[1, -1, \lambda_c]}& (x, -y+\lambda_c(x))\ar@{|->}[r]^-{[1, -1, \lambda_c]} & (x, y-\lambda_c(x)+\lambda_c(x))
        }=(x, y)
    \end{equation*}
    and
    \begin{equation*}
        \xymatrix{
        (x, y)\ar@{|->}[r]^-{[-1, 1, \lambda_c]}& (-x, y+\lambda_c(x))\ar@{|->}[r]^-{[-1, 1, \lambda_c]} & (x, y+\lambda_c(x)+\lambda_c(-x))
        }=(x, y).
    \end{equation*}

    Second, we prove that (1)-(4) are separate conjugacy classes. Since $[1, 1, \lambda_0]$ and $[-1, -1, \lambda_0]$ are central in $\text{Aut}(B_{p, 3})$, they belong to singleton conjugacy classes. Due to commutativity of composition in the first two coordinates for $\mathcal{A}_{p, 3}$-elements, no $[1, -1, \lambda_c]$ and $[-1, 1, \lambda_d]$ can belong to the same conjugacy class. Thus, the proposition is proven if we can show that each $[1, -1, \lambda_c]$ is conjugate $[1, -1, \lambda_0]$ and each $[-1, 1, \lambda_c]$ is conjugate $[-1, 1, \lambda_0]$. Indeed, $[-1, 1, \lambda_{-c/2}][1, -1, \lambda_0][-1, 1, \lambda_{-c/2}]^{-1}=[-1, 1, \lambda_{-c/2}][1, -1, \lambda_0][-1, 1, \lambda_{-c/2}]=[-1, -1, \lambda_{-c/2}][-1, 1, \lambda_{-c/2}]=[1, -1, \lambda_{c/2}+\phantom{.}_{-1}\lambda_{-c/2}]=[1, -1, \lambda_c]$. Likewise,
    $[1, -1, \lambda_{-c/2}][-1, 1, \lambda_0][1, -1, \lambda_{-c/2}]^{-1}=[-1, -1, \lambda_{c/2}][1, -1, \lambda_{-c/2}]=[-1, 1, \lambda_{c/2}+\lambda_{c/2}]=[-1, 1, \lambda_c]$.
\end{proof}

\begin{theorem}\label{Thm:3pLBDS}
    Let $p>3$ be prime. There are $6$ isomorphism classes of LBDS of order $3p$. These classes have the following representatives:
    \begin{itemize}
        \item[(a)] $(\mathbb{Z}/_{3p}, +, 0, \pm1)$;
        \item[(b)] $(B_{p, 3}, \boxplus, (0, 0), [\pm 1, \pm 1, \lambda_0])$; and
        \item[(c)]  $(B_{p, 3}, \boxplus, (0, 0), [\pm1, \mp 1, \lambda_0])$.
    \end{itemize}
\end{theorem}

\begin{remark}\label{Rmk:3pSqMapAnalysis}
    So far, all examples of LBDS constructed have been split extensions involving solutions of type \eqref{Eq:GenofSmSoln} and \eqref{Eq:GenofDerSoln}. However, Theorem \ref{Thm:3pLBDS} provides our first class of examples of LBDS which do not satisfy the hypothesis of Theorem \ref{Thm:LBDSSplitExt}. Take $(B_{p, 3}, \boxplus, (0, 0), [-1, 1, \lambda_0])$. Here, the normal substructure involves a solution upon which $S$ acts by $1$, not $-1$ (cf. the contrapositive of Corollary \ref{Cor:LBDSSPlitExt}). To be sure, for the solution $(B_{p, 3}, \boxplus, (0, 0), [-1, 1, \lambda_0])$, $\text{Sq}_{/^\circ}((i, j))=(0, j)$. This is not a loop endomorphism, for as long as $l\neq0\in \mathbb{Z}/_p$, we have $\text{Sq}_{/^\circ}((1, j)\boxplus(0, l))=(0, j-2l)\neq(0, j+l)=\text{Sq}_{/^\circ}((1, j))\boxplus\text{Sq}_{/^\circ}((0, l))$.
\end{remark}

\subsubsection{Enumerating classes of orders 27 and 81}\label{SubSubSec:Order27_81}
The only orders at which nonassociative Bruck loops are completely accounted for in \texttt{GAP} are $27$ and $81$. Thus, we use the \texttt{LOOPS} package \cite{LOOPS} to enumerate isomorphism classes of LBDS of these orders. By Theorem \ref{Thm:LBDSIsom}, this amounts to counting conjugacy classes of involutions in the automorphism group.

There are $7$ Bruck loops of order 27. In Table \ref{Table:BL_of_order27}, each row corresponds to an order-27 Bruck loop. The first column shows how each loop can be instantiated in \texttt{GAP}, by calling \texttt{LeftBruckLoop(27, i)}. The second column gives isomorphism-invariant information regarding the loop, and in the third column, $n_{\text{CI}}$ stands for the number of conjugacy classes of involutions in the automorphism group of the loop.
\begin{table}[hbt]
\begin{center}
    \begin{tabular}{|c|c|c|}
        \hline
        \texttt{GAP} ID & loop structure & $n_{\text{CI}}$  \\
        \hline
         27/1& $(\mathbb{Z}/_3)^3$ & 4\\
         \hline
         27/2 & $\mathbb{Z}/_3\times\mathbb{Z}/_9$ & 4\\
         \hline
         27/3 &$14$ order-3 els. & 4  \\
         \hline
         27/4 & $20$ order-3 els.  & 4 \\
         \hline
         27/5 & $2$ order-3 els. & 2 \\
         \hline
         27/6 & $8$ order-3 els. & 4 \\
         \hline
         27/7 & $\mathbb{Z}/_{27}$ & 2 \\
         \hline
    \end{tabular}
\end{center}
\vskip2mm
\caption{Enumeration of LBDS of order 27.}\label{Table:BL_of_order27}
\end{table}

Table \ref{Table:BL_of_order81} counts the LBDS coming from each of the 72 Bruck loops of order $81$. At this order, there are no invariants that can succinctly differentiate all 72 loops, hence the lack of a ``loop structure" column.
\begin{table}[hbt]
\begin{center}
    \begin{tabular}{|c|c||c|c||c|c||c|c|}
        \hline
        \texttt{GAP} ID & $n_{\text{CI}}$ & \texttt{GAP} ID & $n_{\text{CI}}$ & \texttt{GAP} ID & $n_{\text{CI}}$ & \texttt{GAP} ID & $n_{\text{CI}}$ \\
        \hline
        81/1 & 5 & 81/19 & 2  & 81/37 & 4 & 81/55 & 4 \\
        \hline
        81/2 & 6 & 81/20 & 2 & 81/38 & 4 & 81/56 & 4 \\
        \hline
        81/3 & 8 & 81/21 & 4 & 81/39 & 2 & 81/57 & 2 \\
        \hline
        81/4 & 8 & 81/22 & 2 & 81/40 & 2 & 81/58 & 4  \\
        \hline
        81/5 & 4 & 81/23 & 4 & 81/41 & 2 & 81/59 & 4 \\
        \hline
        81/6 & 8 & 81/24 & 2 & 81/42 & 2 & 81/60 & 4 \\
        \hline
        81/7 & 8 & 81/25 & 4 & 81/43 & 2 & 81/61 & 2 \\
        \hline
        81/8 & 8 & 81/26 & 4 & 81/44 & 2 & 81/62 & 2 \\
        \hline
        81/9 & 2 & 81/26 & 4 & 81/45 & 2 & 81/63 & 4 \\
        \hline
        81/10 & 4 & 81/28 & 4 & 81/46 & 4 & 81/64 & 4 \\
        \hline
        81/11 & 4 & 81/29 & 4 & 81/47 & 4 & 81/65 & 2 \\
        \hline
        81/12 & 4 & 81/30 & 4 & 81/48 & 4 & 81/66 & 4 \\
        \hline
        81/13 & 2  & 81/31 & 8 & 81/49 & 4 & 81/67 & 4 \\
        \hline
        81/14 & 2 & 81/32 & 4 & 81/50 & 2 & 81/68 & 2 \\
        \hline
        81/15 & 2 & 81/33 & 4 & 81/51 & 3 & 81/69 & 2 \\
        \hline
        81/16 & 2 & 81/34 & 4 & 81/52 & 4 & 81/70 & 2 \\
        \hline
        81/17 & 4 & 81/35 & 4 & 81/53 & 4 & 81/71 & 3  \\
        \hline
        81/18 & 2 & 81/36 & 6 & 81/54 & 4 & 81/72 & 2 \\
        \hline
    \end{tabular}
\end{center}
\vskip2mm
\caption{Enumeration of LBDS of order 81.}\label{Table:BL_of_order81}
\end{table}

Our results are summarized by

\begin{theorem}\label{Thm:2781}
    Up to isomorphism, there are $24$ LBDS of order 27 and $263$ LBDS of order 81.
\end{theorem}

\section{Latin braided triality sets}\label{Sec:LBTS}

\subsection{Latin braided triality sets}\label{SubSec:LBTS}
We will use LBTS as an acronym for Latin, braided triality set.

\begin{lemma}\label{Lmm:Exp3Bruck}
    A Bruck loop of exponent $3$ is a $\CMLtThm$.
\end{lemma}

\begin{proof}
    First, we show that any exponent-$3$ Bruck loop, $Q$, is commutative. By \eqref{Eq:PowerAssoc}, \eqref{Eq:BruckLaw1}, the AIP, and the fact that $2x=-x$, we have $x+(x+y)=2x+y=2(x-y)=x+(-2y+x)=x+(y+x)$; cancel $x$ on the left of the first and last expression, and we have $x+y=y+x$. Therefore, $-(x+y)=-x-y=-y-x$. By Theorem 2.7(vi) in \cite{ROB}, this is enough to conclude that $Q$ is Moufang.
\end{proof}

\begin{theorem}\label{Thm:LBTSEquivConds}
    Let $(Q, r)$ be an LBDS with associated LBDS quasigroup $(Q, \circ, \ld, /^\circ)$ and associated Bruck loop $(Q, +, e)$. The following are equivalent:
    \begin{itemize}
        \item[(a)] $(Q, r)$ is an LBTS;
        \item[(b)] $(Q, +, e)$ is a $\CMLtThm$;
        \item[(c)] $(Q, \circ, \ld, /^\circ)=(Q, \circ, \ld, \circ)$ is a right symmetric quasigroup.
    \end{itemize}
\end{theorem}

\begin{proof}
    $(a)\implies (b)$ By Proposition \ref{Cor:ExponentBruckLoop}, $(Q, +, e)$ has exponent $3$, and so Lemma \ref{Lmm:Exp3Bruck} demands $(Q, +, e)$ is a $\CMLt$.

    $(b)\implies (c)$ Since $(Q, +, e)$ is an IP loop, we have $(y\circ x)\circ x=(-y-x^S)\circ x=-(-y-x^S)-x^S=(y+x^S)-x^S=y.$

    $(c)\implies (a)$ We have to show \eqref{Eq:Tri1} holds by proving
    \begin{equation*}
        (x\circ y)\circ x^S=y^S.
    \end{equation*}

 \noindent Well, in a right-symmetric setting, Lemma \ref{Lmm:DihedLmm2}(c) becomes $x\circ y=y^S\circ x^S$, which implies $(x\circ y)\circ x^S=y^S$. Finally, we must establish \eqref{Eq:Tri2}, which states that
     \begin{equation*}
        (x\circ y)^S=y\circ x.
    \end{equation*}
    Again, Lemma \ref{Lmm:DihedLmm2}(c), ensures this is the case, as $(x\circ y)^S=x^S\circ y^S=x^S/^\circ y^S=y/^\circ x=y\circ x$.
\end{proof}

\begin{corollary}\label{Cor:LBTSQugpIds}
There is a one-to-one correspondence between LBTS and the subvariety of quasigroups generated by
\begin{align*}
    x* yz&=xy* (x\backslash_*x)z\\
    y&=(x\backslash_*x)* xy\\
    y&=yx* x.
\end{align*}
\end{corollary}

\noindent We shall now refer to quasigroups in the variety described by Corollary \ref{Cor:LBTSQugpIds} as \emph{LBTS-quasigroups}.

\begin{corollary}\label{Cor:LBTSCML}
Let $(Q, r)$ be a non-degenerate mapping with $r(x, y)=(x\circ y, x\bullet y)$, and denote $x\ld x$ exponentially as $x^S$. All of the following are equivalent.
\begin{itemize}
    \item[(a)] The pair $(Q, r)$ is an LBTS.
    \item[(b)]  $(Q, \circ, \ld)$ is in fact a right-symmetric LF-quasigroup $(Q, \circ, \ld, \circ)$  in which $x^S\circ(x\circ y)=y$; moreover, $r(x, y)=(x\circ y, x^S)$.
    \item[(c)] For any $e\in \emph{Sq}_\circ(Q)$, the operation $x+y:=e\circ (x^S\circ y)$ endows a $\CMLtThm$ $(Q, +, e)$ such that $S\in \emph{Aut}(Q, +, e)$; moreover, $r(x, y)=(-x-y^S, x^S)$.
\end{itemize}
\end{corollary}

\begin{corollary}\label{Cor:LBTSOrder}
If $(Q, r)$ is a finite LBTS, then there is some $n\geq 0$ such that $|Q|=3^n$.
\end{corollary}

\begin{proof}
The underlying set must play host to a $\CMLt$-structure.  Now, CMLs adhere to both Lagrange's theorem, and to Cauchy's theorem on elements of prime order \cite{CommALoops}.  Thus, $3\mid |Q|$, and any other prime dividing the order of $Q$ would violate the fact that the exponent of our loop is $3$.
\end{proof}

\subsubsection{Multiplicative squaring, substructures, and extension}

We now turn our attention to $\text{Sq}_\circ$, which in the triality setting, coincides with the map $\text{Sq}_{/^\circ}$ of Section \ref{SubSubSec:SquareLBDS}. We now have a stronger version of Proposition \ref{Prop:IdempSubqugp}:

\begin{theorem}\label{Thm:LBTSFactorization}
Let $(Q, +, e, S)$ be an LBTS with associated quasigroup $(Q, \circ, \ld, /^\circ)$.
\begin{itemize}
    \item[(a)] The set of idempotents of $(Q, \circ, \ld, \circ)$ forms a subquasigroup and a subloop that coincides with the set $\{x\in Q\mid x^S=x\}$.
    \item[(b)] The set $Q_e=\{x\in Q\mid \emph{Sq}_\circ(x)=e\}$ forms a subquasigroup of $(Q, \circ, \ld, \circ)$ and a subloop of $(Q, +, e)$ that coincides with the set $\{x\in Q\mid x^S=-x\}$.
    \item[(c)] We have loop and quasigroup factorizations $Q=\emph{Sq}_\circ(Q)+Q_e=\emph{Sq}_\circ(Q)\circ Q_e$.
\end{itemize}
\end{theorem}

\begin{proof}
   Statements (a) and (b) follow immediately from Proposition \ref{Prop:IdempSubqugp}.\\
    (c) The CML identity \eqref{Eq:Manin} demands
    \begin{align*}
        x&=x+(-x^S+x^S)=(-x-x)+(-x^S+x^S)\\
        &=(-x-x^S)+(-x+x^S)
    \end{align*}
    for all $x\in Q.$ By (a), $-x-x^S\in \text{Sq}_\circ(Q)$, as $(-x-x^S)^S=-x^S-x=-x-x^S$. By (b), $-x+x^S\in Q_e,$ as $(-x+x^S)^S=-x^S+x=x-x^S=-(-x+x^S)$. Thus, $Q=\text{Sq}_\circ(Q)+Q_e$. A routine calculation verifies that $$x=(-x-x^S)+(-x+x^S)=((-x)\circ (-x))\circ(x\circ(-x)),$$ and thus $Q=\text{Sq}_\circ(Q)\circ Q_e$.
\end{proof}

The following are immediate consequences of Theorem \ref{Thm:LBDSSplitExt}.

\begin{theorem}\label{Thm:LBTSSplitExt}
Let $(Q, +, e, S)$ be an LBTS with associated quasigroup $(Q, \circ, \ld, \circ)$. If $\emph{Sq}_{\circ}$ is an endomorphism of the LBTS, then we have a split extension

\begin{equation}\label{Eq:LBTSSPlitExt}
    \{e\}\to Q_e\to Q\to \emph{Sq}_{\circ}(Q)\to \{e\}
\end{equation}
in both the category of LBTS-quasigroups and the category of commutative Moufang loops of exponent $3$.
\end{theorem}

\begin{corollary}\label{Cor::LBTSSmplitExt}
    Let $(Q, +, e, S)$ be an LBTS such that $\emph{Sq}_{\circ}$ is endomorphic;
    \begin{itemize}
        \item[(a)] the $\CMLtThm$ automorphism $S$ is upper-triangularizable in the sense that $Q$ is a split extension of a normal subloop upon which $S$ acts as $-1$ with a subloop upon which $S$ acts as $1.$
        \item[(b)] $(Q, +, e, S)$ is a split extension of a solution of type \eqref{Eq:SMith'sSol} with one of type \eqref{Eq:CMLDerSol}.
    \end{itemize}
\end{corollary}

\subsection{LBTS of orders 3, 9, 27, 81}\label{SubSec:SmallOrderLBTS}
There is one nonassociative $\CMLt$ of order $81$ \cite[Thm.~9.2]{KepkaNemec}. It has the following representation in terms of the ring $\mathbb{Z}/_3[\omega]$, where $\omega$ is a primitive third root of unity (cf. \cite[Sec.~5.2]{No22}).

\begin{definition}\label{Def:L3}
    Let $C_3^2=\{b^mc^n\mid m, n\in \mathbb{Z}/_3\}$ denote the rank-2 elementary abelian $3$-group with identity abbreviated to $e$. Denote $L_3= C_3^2\times \mathbb{Z}/_3[\omega]$. Define the operation
    \begin{equation}
        (b^{m_1}c^{n_1}, q+r\omega)\boxplus(b^{m_2}c^{n_2}, s+t\omega)=(b^{m_1+m_2}c^{n_1+n_2}, \omega^\nu(q+r\omega)+\omega^{-\nu}(s+t\omega)),
    \end{equation}
    where $\nu=m_1n_2-n_1m_2$. Then $(L_3, \boxplus, (e, 0))$ is a nonassociative $\CMLt$ of order $81$. Inversion coincides with inversion in the underlying abelian group. Its associative center is given by the set $Z(L_3)=\{(e, a-a\omega)\in L_3\mid a\in \mathbb{Z}/_3\}$.
\end{definition}

\begin{lemma}\label{Lmm:L3LBTS}
The four LBTS $(L_3, \boxplus, (e, 0), S_1), \dots, (L_3, \boxplus, (e, 0), S_4)$, where
    \begin{align*}
        S_1& \text{ is the identity map},\\
        S_2& \text{ is inversion in } L_3,\\
        S_3&:(b^mc^n, q+r\omega)\mapsto(b^mc^n, -q-r\omega), \text{ and}\\
        S_4&:(b^mc^n, r+s\omega)\mapsto (b^{-m}c^{-n}, r+s\omega).
    \end{align*}
    are pairwise non-isomorphic.
\end{lemma}

\begin{proof}
Since $S_1$ and $S_2$ are central in $\text{Aut}(L_3)$, they are each pairwise non-isomorphic to any LBTS over $L_3$. Now, we just need to show $S_3$ and $S_4$ correspond to non-isomorphic LBTS. This follows from the fact that in the LBTS-quasigroup corresponding to $S_3$, the multiplicative squaring map is endomorphic, while in the LBTS-quasigroup of $S_4$, it is not. Showing that $(L_3, \boxplus, (e, 0), S_4)$ fails to have endomorphic multiplicative squaring map works just like our argument in Remark \ref{Rmk:3pSqMapAnalysis} that showed the right divisional squaring map in $(B_{p, 3}, \boxplus, (0, 0), [-1, 1, \lambda_0])$ was not endomorphic.
\end{proof}

\begin{theorem}\label{Thm:LBTS81}
Let $Q_1$ denote the order-3 LBTS $(\mathbb{Z}/_3, +, 0, 1)$ and $Q_{-1}$ the order-$3$ LBTS $(\mathbb{Z}/_3, +, 0, -1)$. If $(Q, +, e, S)$ is a nontrivial LBTS such that $|Q|\leq 81$, then it is isomorphic to exactly one of the following:
    \begin{itemize}
        \item[(a)] \underline{order 3}: $Q_1,$ $Q_{-1}$
        \item[(b)] \underline{order 9}: $(Q_1)^2,$ $(Q_{-1})^2$, $Q_1\times Q_{-1}$
        \item[(c)] \underline{order 27}: $(Q_1)^3$, $(Q_{-1})^3$, $(Q_1)^2\times Q_{-1}$, $(Q_{-1})^2\times Q_1$
        \item[(d)] \underline{order 81, associative}: $(Q_1)^4, (Q_{-1})^4$, $(Q_1)^3\times Q_{-1}$, $(Q_{-1})^3\times Q_1$, $(Q_1)^2\times (Q_{-1})^2$,
        \item[(e)] \underline{order 81, nonassociative}: $(L_3, \boxplus, (e, 0), S_i)$ where $1\leq i\leq 4$ and each $S_i$ comes from Lemma \ref{Lmm:L3LBTS}.
    \end{itemize}
\end{theorem}

\begin{proof}
    If the underlying loop is associative, then conjugacy classes of involutions are determined by the rational canonical form over $\mathbb{Z}/_3$. This accounts for items (a)-(d).  Since $(L_3, \boxplus, (e, 0))$ is the sole nonassociative $\CMLt$ of order $81$, all remaining LBTS are defined over this loop. As we saw in Lemma \ref{Lmm:L3LBTS}, the $(L_3, \boxplus, (e, 0), S_i)$'s are pairwise non-isomorphic. Moreover, one may use \texttt{GAP} to compute the automorphism group of $L_3$ and its conjugacy classes; there are four classes of involutions, and, thus, the $S_i$'s account for all LBTS occurring over $L_3.$
\end{proof}

\subsection*{Acknowledgement}
The authors would like to thank Michael Kinyon for help with \texttt{Mace4} and for pointing out the fact that Bruck loops would be a unifying class of loop isotopes for LBDS.

\appendix
\section{Some proofs}\label{appendix:human_proof}
\begin{lemma}\label{Lmm:RQ}
Let $(Q, r)$ be an LBDS. Then $(Q, r)$ is a non-degenerate solution.
\end{lemma}
\begin{proof}
We will show that, for each $y\in Q$, the right translation map $R^\bullet_y\colon Q\to Q$, $R^\bullet_y(x)=x\bullet y$ is a bijection. Let $x\bullet y=z\bullet y$ for some $x,z\in Q$. By \eqref{Eq:Di2} one has: $(x\bullet y)\circ(x\circ y)=(z\bullet y)\circ(z\circ y)$ and, since $(Q,\circ)$ is a left quasigroup, $x\circ y=z\circ y$. Now, since $(Q,\circ)$ is a right quasigroup, too, $x=z$ and $R^\bullet_y$ is an injection. To show that it is \emph{onto}, consider the element $a=w\bullet(w\setminus_\circ y)\in Q$ for some $w\in Q$. Then
\begin{align*}
(w\bullet(w\setminus_\circ y))\bullet y=(w\bullet(w\setminus_\circ y))\bullet (w\circ(w\setminus_\circ y))=w.
\end{align*}
The first equality is due to the fact that $(Q,\circ)$ is a left quasigroup, so $w\circ(w\setminus_\circ y)=y$ and the second one, by \eqref{Eq:Di1}. This shows that for each $w\in Q$ there exists $a\in Q$ such that $w=R^\bullet_y(a)$ and $R^\bullet_y$ is a surjection.
\end{proof}
Consider diagonal mappings $x^S=x\setminus_\circ x$ and $x^T=x/^\bullet x$.
\begin{lemma}\label{Lmm:STmutinv}
Let $(Q, r)$ be an LBDS. The mappings $S$ and $T$ are mutually inverse.
\end{lemma}
\begin{proof}
Substituting $y$ by $x\setminus_\circ x$ in \eqref{Eq:YB1} one obtains
\begin{align*}
x\circ ((x\setminus_\circ x)\circ z)&=(x\circ(x\setminus_\circ x))\circ ((x\bullet (x\setminus_\circ x))\circ z)\\
&=x\circ ((x\bullet (x\setminus_\circ x))\circ z).
\end{align*}
Now, first applying the left cancellative law and then the right one for a quasigroup $(Q,\circ)$, one has
$x\setminus_\circ x=x\bullet (x\setminus_\circ x)$. The latter is equivalent to the identity \eqref{E:Sta_6}, and to \eqref{E:Sta_7} (see Remark \ref{Rmk:Birack}).

Translating \eqref{E:Sta_6} and \eqref{E:Sta_7} to the language of mappings $S$ and $T$ one gets that they are mutually inverse.
\end{proof}
\begin{lemma}\label{Lmm:DihedSinv}
The diagonal mapping $S$ is involutive. In particular, $S=T$.
\end{lemma}
\begin{proof}
Substituting $y$ by $x\setminus_\circ x$ in \eqref{Eq:Di2} one obtains $(x\bullet (x\setminus_\circ x))\circ(x\circ (x\setminus_\circ x))=x\setminus_\circ x$. Hence, $(x\setminus_\circ x)\circ x=x\setminus_\circ x$ and finally $x=(x\setminus_\circ x)\setminus_\circ(x\setminus_\circ x)$.
\end{proof}

\begin{lemma}\label{Lmm:DihedIdent}
 In any LBDS $(Q, r)$, all of the following hold:
    \begin{itemize}
        \item[(a)] $x\bullet (x\bullet y)=(x\bullet y)\circ x$.
        \item[(b)] $(x\bullet y)\bullet x=x\circ(x\bullet y)$.
        \item[(c)] $(x\bullet y)\bullet (y\setminus_\circ y)=x$.
    \end{itemize}
\end{lemma}
\begin{proof}

(a) By \eqref{Eq:Di2} and \eqref{Eq:YB1} we have:
\begin{align*}
 y=(x\bullet y)\circ (x\circ y)=((x\bullet y)\circ x)\circ (((x\bullet y)\bullet x)\circ y).
\end{align*}
On the other hand, by \eqref{Eq:Di2}, we have also
\begin{align*}
 (((x\bullet y)\bullet x)\bullet y)\circ (((x\bullet y)\bullet x)\circ y) =y.
\end{align*}
And, since $(Q,\circ)$ is a right quasigroup, $(x\bullet y)\circ x=((x\bullet y)\bullet x)\bullet y$ follows. Now, by \eqref{Eq:YB3} and \eqref{Eq:Di1} $((x\bullet y)\bullet x)\bullet y=x\bullet (x\bullet y)$.

Substituting in (a) $x$ by $x\bullet y$ and $y$ by $x\circ y$ we obtain (b).

First, to prove (c), substitute in \eqref{Eq:YB1}
 $x$ by $x\bullet y$, $y$ by $y\setminus_\circ y$ and $z$ by $y$:
 \begin{align*}
 &(x\bullet y)\circ (y\setminus_\circ y)= ((x\bullet y)\circ (y\setminus_\circ y))\circ(((x\bullet y)\bullet (y\setminus_\circ y)) \circ y).
 \end{align*}
 Since $(Q,\circ)$ is a left quasigroup, we obtain then
 \begin{align*}
 ((x\bullet y)\bullet (y\setminus_\circ y)) \circ y=((x\bullet y)\circ (y\setminus_\circ y))\setminus_\circ ((x\bullet y)\circ (y\setminus_\circ y))=((x\bullet y)\circ (y\setminus_\circ y))^S.
 \end{align*}
 By \eqref{Eq:YB1} we have $(x\bullet y)\circ z=(x\circ y)\setminus_\circ (x\circ (y\circ z))$. Hence,
 $(x\bullet y)\circ (y\setminus_\circ y)=(x\circ y)\setminus_\circ (x\circ y)=(x\circ y)^S$. Since the squaring map $S$ is involutive, we get
 $((x\bullet y)\bullet (y\setminus_\circ y)) \circ y=x\circ y$  and, since $(Q,\circ)$ is a right quasigroup,
 $x=(x\bullet y)\bullet (y\setminus_\circ y)$.
 \end{proof}
\begin{lemma}\label{Lmm:ident}
Let $(Q, r)$ be an LBDS. Then $(x\circ y)\bullet z= (x\bullet z)\circ (y\bullet (x\circ z))$.
\end{lemma}
\begin{proof}
By \eqref{Eq:Di1} and \eqref{Eq:Di2}, $y\bullet(x\circ z)=((x\bullet y)\circ (x\circ y))\bullet (((x\bullet y)\bullet (x\circ y))\circ z)$. Applying now \eqref{Eq:YB2} to the latter, we obtain the following
\begin{align*}
y\bullet(x\circ z)=\underbrace{((x\bullet y)\bullet ((x\circ y)\circ z))}_{w}\circ \underbrace{((x\circ y)\bullet z)}_d.
\end{align*} Now, consider the expression
$(w\bullet d)\circ(w\circ d)=d$. By \eqref{Eq:YB3} and \eqref{Eq:Di1}  $w\bullet d=((x\bullet y)\bullet(x\circ y))\bullet z=x\bullet z$ and $(x\bullet z)\circ (y\bullet (x\circ z))=(x\circ y)\bullet z$.
\end{proof}
\begin{lemma}
Let $(Q, r)$ be an LBDS. Then $x\bullet y=x^S$.
\end{lemma}

\begin{proof}
We substitute $x$ by $x\bullet y$ and $y$ by $x$ in the identity (c) from Lemma \ref{Lmm:DihedIdent}. We apply also the identity (b) from the same lemma. Hence, $x\bullet y=((x\bullet y)\bullet x)\bullet (x \setminus_\circ x)=(x\circ (x\bullet y))\bullet (x \setminus_\circ x)$. By Lemma \ref{Lmm:ident}, $x\bullet y=(x \setminus_\circ x)\circ (x\circ (x\bullet y))$.  By \eqref{Eq:Di2}, $x\bullet y=(x\bullet (x\bullet y))\circ(x\circ (x\bullet y))$, and hence $x \setminus_\circ x=x\bullet(x\bullet y)=(x\bullet y)\circ x$. Now, since  $(x \setminus_\circ x)\circ x=x \setminus_\circ x=(x\bullet y)\circ x$,  we obtain  $x\bullet y=x\setminus_\circ x$.
   \end{proof}

\end{document}